\documentclass[pdftex]{amsart}

\address{Department of Mathematics, Tokyo Institute of Technology, 2-12-1 Ookayama, Meguro-ku, Tokyo, 152-8551, Japan}
\email{isoshima.t.aa@m.titech.ac.jp}

\usepackage{amsmath,amssymb}
\usepackage{amsfonts}
\usepackage{comment}
\usepackage{graphicx}
\usepackage[all]{xy}

\usepackage{hyperref}
\usepackage{tikz}
\usetikzlibrary{intersections,calc,arrows}
\usetikzlibrary{cd}
\usetikzlibrary{decorations.markings}
\usetikzlibrary{arrows.meta}

\usepackage{here}
\usepackage{caption}

\usepackage[final]{pdfpages}
\usepackage{amsthm}

\theoremstyle{plain}
\newtheorem{thm}{Theorem}[section]
\newtheorem{prop}[thm]{Proposition}
\newtheorem{lem}[thm]{Lemma}
\newtheorem{cor}[thm]{Corollary}
\newtheorem{exam}[thm]{Example}
\newtheorem*{thm*}{Theorem}
\newtheorem*{cor*}{Corollary}

\theoremstyle{definition}
\newtheorem{dfn}[thm]{Definition}
\newtheorem{rem}[thm]{Remark}

\newtheorem{que}[thm]{Question}

\newtheorem*{que*}{Question}
\newtheorem*{con*}{Conjecture}

\begin{document}

\title{Trisections obtained by trivially regluing surface-knots}
\author{Tsukasa Isoshima}
\date{}

\begin{abstract}
Let $S$ be a $P^2$-knot which is the connected sum of a 2-knot with normal Euler number 0 and an unknotted $P^2$-knot with normal Euler number $\pm2$ in a closed 4-manifold $X$ with trisection $T_{X}$. Then, we show that the trisection of $X$ obtained by the trivial gluing relative trisections of $\overline{\nu(S)}$ and $X-\nu(S)$ is diffeomorphic to a stabilization of $T_{X}$. It should be noted that this result is not obvious since boundary-stabilizations introduced by Kim and Miller are used to construct a relative trisection of $X-\nu(S)$. As a corollary, if $X=S^4$, the resulting trisection is diffeomorphic to a stabilization of the genus 0 trisection of $S^4$. This result is related to the conjecture that is a 4-dimensional analogue of Waldhausen's theorem on Heegaard splittings. 
\end{abstract}

\maketitle

\section{Introduction}\label{sec:intro}

In 2012, Gay and Kirby \cite{GK} introduced the notion of a trisection of a 4-manifold, which is an analogue of a Heegaard splitting of a 3-manifold. A trisection of a 4-manifold with boundary is called a relative trisection. Meier and Zupan \cite{MZ2} introduced the notion of a bridge trisection of a surface-knot, which is an analogue of a bridge decomposition of a classical knot. A surface-knot can be put in a nice position in a 4-manifold, called a bridge position, such that the surface-knot is trisected according to a trisection of the 4-manifold.

Let $T=(X_1 , X_2 , X_3)$ be a trisection of a 4-manifold $X$, namely, $X=X_1 \cup X_2 \cup X_3$ and each $X_i$ is a 4-dimensional 1-handlebody. For a 2-knot $K$ in $X$ which is in 1-bridge position, the decomposition of $X-\nu(K)$ into the union of three $X_i-\nu(K)$'s is a relative trisection of $X-\nu(K)$, where $\nu(K)$ is an open tubular neighborhood of $K$. On the other hand, for a surface-knot $S$ in $X$ which is not a 2-knot, the decomposition of $X-\nu(S)$ is never a relative trisection of $X-\nu(S)$. Kim and Miller \cite{KM} introduced a new technique, called a boundary-stabilization, to change the above decomposition of $X-\nu(S)$ into a relative trisection.

We can construct a new trisection of $X=\overline{\nu(S)} \cup_{id} X-\nu(S)$ by gluing a relative trisection of $\overline{\nu(S)}$ and that of $X-\nu(S)$ constructed above using a gluing technique given by Castro and Ozbagci \cite{CO}. In this section, the new trisection is called a trisection obtained by trivially gluing $\nu(S)$ and $X-\nu(S)$. This trisection and $T$ are stably diffeomorphic (resp. stably isotopic), namely, they are diffeomorphic (resp. isotopic) after finitely many stabilizations. However, it is not obvious whether this trisection is diffeomorphic, especially isotopic, to a stabilization of $T$ since when we construct a relative trisection of $X-\nu(S)$ from the union of three $X_i-\nu(S)$'s, we use boundary-stabilizations as mentioned above. Thus, we can think about the following question.

\begin{que*}[Question \ref{que:original question}]
Let $S$ be a surface-knot in a  closed 4-manifold $X$ with trisection $T$. Is a trisection obtained by trivially gluing $\nu(S)$ and $X-\nu(S)$ diffeomorphic, especially isotopic, to a stabilization of $T$? In particular, if $X=S^4$, does this hold?
\end{que*}

The Price twist is a surgery along a $P^2$-knot $P$ in a 4-manifold $X$, which yields at most three different 4-manifolds, namely, $X$, $\Sigma_P({X})$ and a  non-simply connected 4-manifold $\tau_P({X})$. The closed 4-manifold $\Sigma_P({S^4})$ is a homotopy 4-sphere. In this paper, we call the twist having $X$ the trivial Price twist. Kim and Miller \cite{KM} constructed trisections obtained by the Price twist by attaching a relative trisection of $\overline{\nu(P)}$ obtained from its Kirby diagram to a relative trisection of $X-\nu(P)$ constructed by a boundary-stabilization. 

In this paper, we show the following theorem for Question \ref{que:original question}. Note that a trisection obtained by the trivial Price twist along $S$ corresponds to that obtained by trivially gluing a relative trisection of $\overline{\nu(S)}$ and that of $X-\nu(S)$.

\begin{thm*}[Theorem \ref{main theorem}]
Let $X$ be a closed 4-manifold and $S$ the connected sum of a 2-knot $K$ with normal Euler number 0 and an unknotted $P^2$-knot with normal Euler number $\pm2$ in $X$. Also let $T_{(X,S)}$ be a bridge trisection of $(X,S)$ and $T_X$ the underlying trisection. Suppose that $S$ is in bridge position with respect to $T_X$. Also let $T_{X}^{'}$ be the underlying trisection of the bridge trisection obtained by meridionally stabilizing $T_{(X,S)}$ so that $S$ is in 2-bridge position with respect to $T_{X}^{'}$. Then, the trisection $T_S$ obtained by the trivial Price twist along $S$ is diffeomorphic to a stabilization of $T_{X}^{'}$. In particular, the trisection $T_S$ is diffeomorphic to a stabilization of $T_{X}$.
\end{thm*}

In the proof of Theorem \ref{main theorem}, we will perform handle slides and destabilizations many times (see also \cite{Na}).

A $P^2$-knot $S$ in $S^4$ is said to be of Kinoshita type if $S$ is the connected sum of a 2-knot and an unknotted $P^2$-knot. It is conjectured that every $P^2$-knot in $S^4$ is of Kinoshita type (see Remark \ref{rem:Kisoshita conjecture}). 

\begin{cor*}[Corollary \ref{main corollary}]
For each $P^2$-knot $S$ in $S^4$ that is of Kinoshita type, the trisection obtained by the trivial Price twist along $S$ is diffeomorphic to a stabilization of the genus 0 trisection of $S^4$.
\end{cor*}

This implies that if any two diffeomorphic trisections of $S^4$ are isotopic, the resulting trisection gives a positive evidence to the conjecture that is a 4-dimensional analogue of Waldhausen's theorem on Heegaard splittings.

\begin{con*}[\cite{MSZ}]
Every trisection of $S^4$ is isotopic to either the genus 0 trisection or its stabilization.
\end{con*}

\section*{Organization}
In Section \ref{sec:preliminaries}, we review trisections, relative trisections and bridge trisections. In Section \ref{sec:Price twist}, we recall a surgery along a $P^2$-knot in a 4-manifold, called the Price twist and provide a topic related to a trisection obtained by the Price twist. In Section \ref{sec:boundary-stabilization}, we review the definition of a boundary-stabilization and the way of constructing a relative trisection of the complement of a surface-knot. Finally, in Section \ref{sec:main theorem}, we raise a question on a stabilization of a trisection obtained by the trivial regluing of a surface-knot and prove our main theorem and its corollary related to the conjecture that is a 4-dimensional analogue of Waldhausen's theorem on Heegaard splittings. 

\section*{Acknowledgement}

The author would like to thank his supervisor Hisaaki Endo for his helpful comments on this research and careful reading of this paper. He also would like to thank Maggie Miller for her helpful comments on his question and David Gay for his advice on our main theorem. 

\section{Preliminaries}\label{sec:preliminaries}

In this paper, we assume that  4-manifolds are compact, connected, oriented, and smooth unless otherwise stated and a surface-knot in a 4-manifold is a closed surface smoothly embedded in the 4-manifold.

\subsection{Trisections of 4-manifolds}
In this subsection, we review a definition and properties of trisections of closed 4-manifolds introduced in \cite{GK}. Let $g$, $k_1$, $k_2$ and $k_3$ be integers satisfying $0 \le k_1,k_2,k_3 \le g$. 

\begin{dfn}\label{def:trisection}
Let $X$ be a closed 4-manifold. A $(g;k_1,k_2,k_3)$-\textit{trisection} of $X$ is a decomposition $X=X_1 \cup X_2 \cup X_3$ into three submanifolds $X_1,X_2,X_3$ of $X$ satisfying the following conditions:
\begin{itemize}
\item For each $i=1,2,3$, there exists a diffeomorphism $\phi_i \colon X_i \to Z_{k_i}$, where $Z_{k_i} = \natural_{k_i}S^1 \times D^3$.
\item For each $i=1,2,3$, $\phi_i(X_i \cap X_{i-1}) = Y_{k_i,g}^{-}$ and  $\phi_i(X_i \cap X_{i+1}) = Y_{k_i,g}^{+}$, where $Y_{k_i,g}^{\pm}$ is the genus $g$ Heegaard splitting $\partial{Z_{k_i}} = Y_{k_i,g}^{-} \cup Y_{k_i,g}^{+}$ of $\partial{Z_{k_i}}$ obtained by stabilizing the standard genus $k_i$ Heegaard splitting of $\partial{Z_{k_i}}$ $g-k_i$ times. 
\end{itemize} 
\end{dfn}

Note that when $X$ admits a trisection $X=X_1 \cup X_2 \cup X_3$, we call the 3-tuple $T=(X_1, X_2, X_3)$ also a trisection of $X$. If $k_1=k_2=k_3=k$, the trisection is called a \textit{balanced} trisection, or a $(g,k)$-trisection; if not, it is called an \textit{unbalanced} trisection. For a $(g,k)$-trisection, since $\chi(X)=2+g-3k$, we simply call the trisection a \textit{genus} $g$ trisection. For example, the 4-sphere $S^4$ admits the $(0,0)$-trisection, namely genus 0 trisection. 

For a trisection $(X_1,X_2,X_3)$, let $H_\alpha=X_3 \cap X_1$, $H_\beta=X_1 \cap X_2$ and $H_\gamma=X_2 \cap X_3$. Then, the trisection is uniquely determined from $H_\alpha \cup H_\beta \cup H_\gamma$ \cite{LP}. The union $H_\alpha \cup H_\beta \cup H_\gamma$ is called the \textit{spine}.


Given a trisection, we can define its diagram, called a trisection diagram. Note that from the definition, we see that the triple intersection $X_1 \cap X_2 \cap X_3$ is an oriented closed surface $\Sigma_g$ of genus $g$.

\begin{dfn}
Let $\Sigma$ be a compact, connected, oriented surface, and $\delta$, $\epsilon$ collections of disjoint simple closed curves on $\Sigma$. The 3-tuples $(\Sigma, \delta, \epsilon)$ and $(\Sigma, \delta^{'}, \epsilon^{'})$ are said to be \textit{diffeomorphism and handleslide equivalent} if there exists a self diffeomorphism $h$ of $\Sigma$ such that $h(\delta)$ and $h(\epsilon)$ are related to $\delta^{'}$ and $\epsilon^{'}$ by a sequence of handleslides, respectively.
\end{dfn}

\begin{dfn}
A $(g;k_1,k_2,k_3)$-\textit{trisection diagram} is a 4-tuple $(\Sigma_g,\alpha,\beta,\gamma)$ satysfying the following conditions:
\begin{itemize}
\item $(\Sigma_g,\alpha,\beta)$ is diffeomorphism and handleslide equivalent to the standard genus $g$ Heegaard diagram of $\#_{k_1}S^1 \times S^2$.
\item $(\Sigma_g,\beta,\gamma)$ is diffeomorphism and handleslide equivalent to the standard genus $g$ Heegaard diagram of $\#_{k_2}S^1 \times S^2$.
\item $(\Sigma_g,\gamma,\alpha)$ is diffeomorphism and handleslide equivalent to the standard genus $g$ Heegaard diagram of $\#_{k_3}S^1 \times S^2$.
\end{itemize}
\end{dfn}

Figure \ref{fig:Heegaarddiagramfork_iS^1×S^2} describes the standard genus $g$ Heegaard diagram of $\#_{k_i}S^1 \times S^2$.

Note that given a trisection diagram $(\Sigma_g,\alpha,\beta,\gamma)$, $\alpha$, $\beta$ and $\gamma$ are respectively indicated by red, blue and green curves as in Figure \ref{fig:trisection diagram of CP^2}.

\begin{figure}[h]
\begin{center}
\includegraphics[width=13cm, height=7cm, keepaspectratio, scale=1]{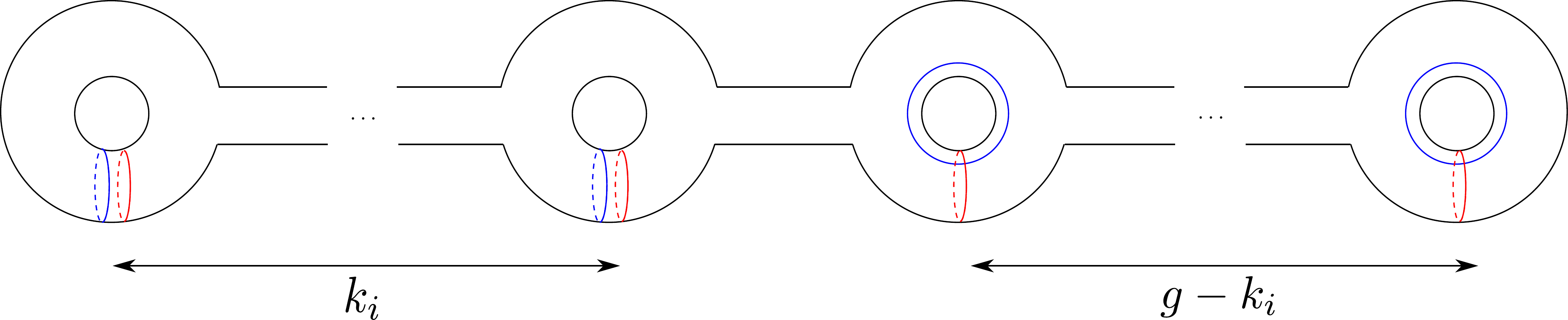}
\end{center}
\caption{The standard genus $g$ Heegaard diagram of $\#_{k_i}S^1 \times S^2$. }
\label{fig:Heegaarddiagramfork_iS^1×S^2}
\end{figure}

\begin{exam}
Figure \ref{fig:trisection diagram of CP^2}  is a $(1,0)$-trisection diagram of $\mathbb{C}P^2$ (see also Figure \ref{fig:dptd of CP^2 and CP^1}).
\end{exam}

\begin{figure}[h]
\begin{center}
\includegraphics[width=5.5cm, height=7cm, keepaspectratio, scale=1]{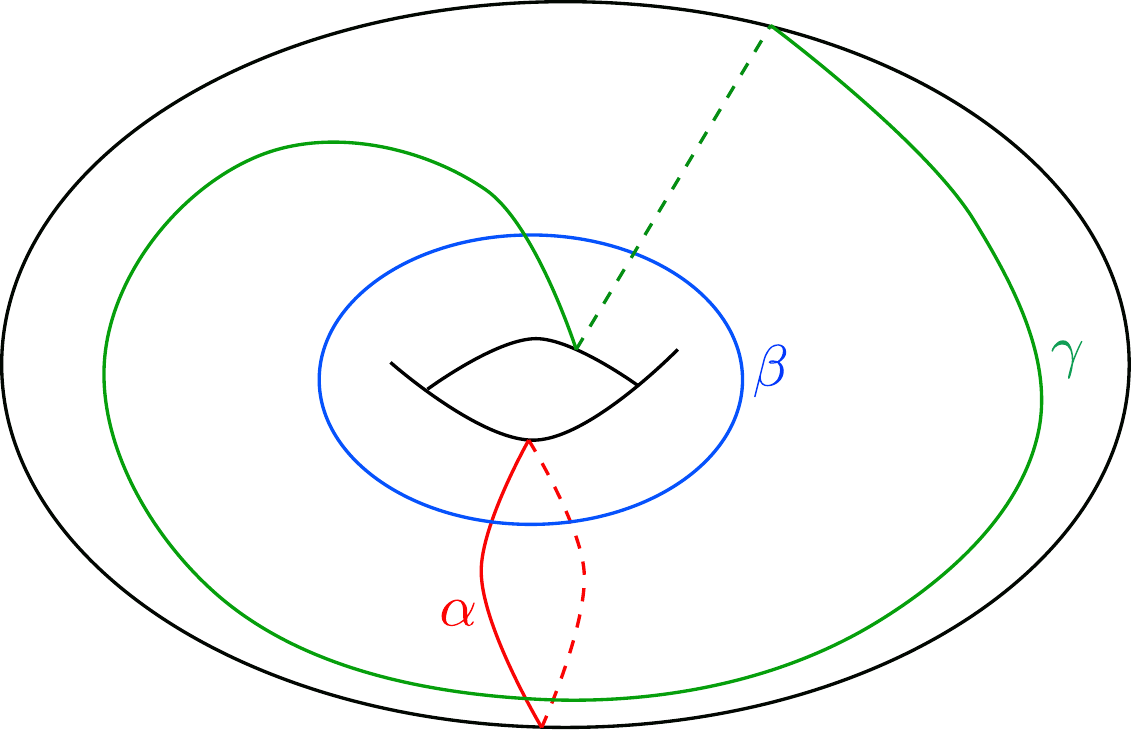}
\end{center}
\caption{A $(1,0)$-trisection diagram of $\mathbb{C}P^2$. }
\label{fig:trisection diagram of CP^2}
\end{figure}

\begin{dfn}[\cite{Is}]
Let $X$ be a closed 4-manifold, and $T=(X_1,X_2,X_3)$ and $T^{'}=(X_1^{'},X_2^{'},X_3^{'})$ trisections of $X$. We say that $T$ and $T^{'}$ are \textit{diffeomorphic} if there exists a diffeomorphism $h \colon X \to X$ such that $h(X_i)=X_i^{'}$ for each $i=1,2,3$. We say that $T$ and $T^{'}$ are \textit{isotopic} if there exists an isotopy $\{h_t\}_{t \in [0,1]}$ of $X$ such that $h_0=id$ and $h_1(X_i)=X_i^{'}$ for each $i=1,2,3$.
\end{dfn}

Note that $T$ and $T^{'}$ are diffeomorphic if and only if trisection diagrams of $T$ and $T^{'}$ are related by handle slides on the same color curves and diffeomorphisms of a surface.

As with the stabilization for a Heegaard splitting, we can define a stabilization for a trisection. 

\begin{dfn}
Let $(X_1,X_2,X_3)$ be a trisection and $C$ a boundary-parallel arc properly embedded in $X_i \cap X_j$. We define $X_i^{'}$, $X_j^{'}$, and $X_k^{'}$ as follows, where $\{i,j,k\}=\{1,2,3\}$.
\begin{itemize}
\item $X_i^{'}=X_i-\nu(C)$,
\item $X_j^{'}=X_j-\nu(C)$,
\item $X_k^{'}=X_k \cup \overline{\nu(C)}$.
\end{itemize}
The replacement of $(X_1,X_2,X_3)$ by $(X_1^{'},X_2^{'},X_3^{'})$ is said to be the $k$-\textit{stabilization}.
\end{dfn}

Note that the stabilization does not depend on the choice of an arc since any two boundary-parallel arcs in a 3-dimensional 1-handlebody are isotopic.

We can define a stabilization for a trisection using its trisection diagram.

\begin{dfn}
Let $(\Sigma,\alpha,\beta,\gamma)$ be a trisection diagram. The diagram obtainted by connect-summing $(\Sigma,\alpha,\beta,\gamma)$ with one of three diagrams depicted in Figure \ref{fig:stabilizationfordiagram} is called the \textit{stabilization} of $(\Sigma,\alpha,\beta,\gamma)$.
\end{dfn}

\begin{figure}[h]
\begin{center}
\includegraphics[width=12cm, height=7cm, keepaspectratio, scale=1]{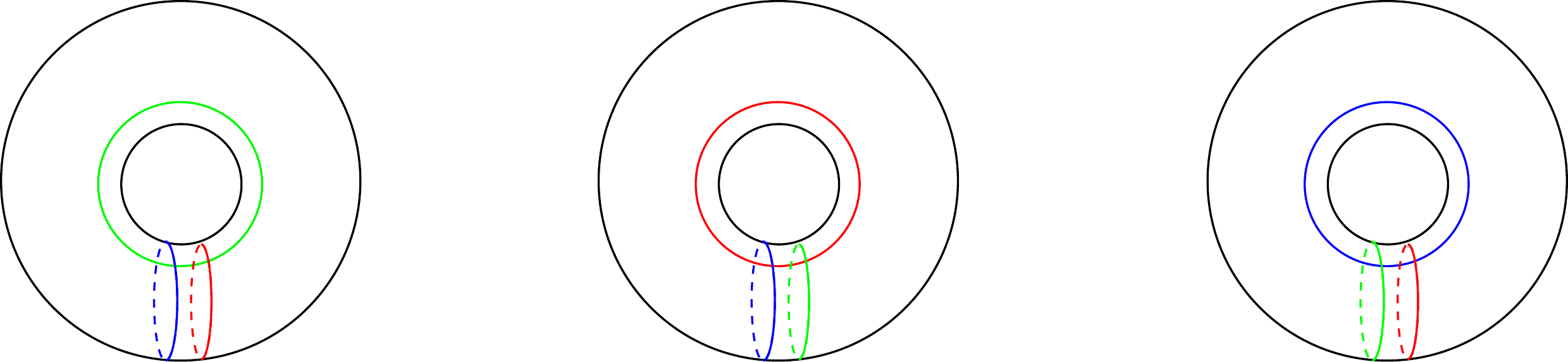}
\end{center}
\caption{The unbalanced trisection diagrams of $S^4$. }
\label{fig:stabilizationfordiagram}
\end{figure}

The diagrams in Figure \ref{fig:stabilizationfordiagram} are $(1;1,0,0),(1;0,1,0),(1;0,0,1)$-trisection diagrams of $S^4$ from left to right. Note that for a $(g;k_1,k_2,k_3)$-trisection diagram $(\Sigma,\alpha,\beta,\gamma)$, the diagram obtained by connect-summing $(\Sigma,\alpha,\beta,\gamma)$ with the leftmost (resp. middle, resp. rightmost) diagram in Figure \ref{fig:stabilizationfordiagram} is a $(g+1;k_{1}+1,k_2,k_3)$ (resp. $(g+1;k_{1},k_{2}+1,k_3)$, resp. $(g+1;k_{1},k_2,k_{3}+1)$)-trisection diagram. 
Given a trisection diagram $(\Sigma, \alpha, \beta, \gamma)$, we can define a closed 4-manifold $X(\Sigma, \alpha, \beta, \gamma)$ as follows: We attach 2-handles to $\Sigma \times D^2$ along $\alpha \times \{1\}$, $\beta \times \{e^\frac{2{\pi}i}{3}\}$, and $\gamma \times \{e^\frac{4{\pi}i}{3}\}$, where the framing of each 2-handle is the surface framing. Then, we attach 3, 4-handles. Note that the way of attaching 3, 4-handles is unique up to diffeomorphism \cite{LP}.

Gay and Kirby \cite{GK} showed that every closed 4-manifold $X$ admits a trisection with nice handle decomposition. Moreover, they showed that any two trisections of a fixed closed 4-manifold are stably isotopic. Namely, they are isotopic after finitely many stabilizations. Note that they proved it in the balanced case. In general, an $i$-stabilized trisection is not isotopic to a $j$-stabilized trisection when $i\not=j$ \cite{MSZ}. 



For more details on trisections of closed 4-manifolds, see \cite{GK}.

\subsection{Relative trisections}

In this subsection, we review trisections of 4-manifolds with boundary, called relative trisections. Before the definition, we introduce some notations.

Let $g$, $k$, $p$ and $b$ be non-negative integers with $b \ge 1$ and $g+p+b-1 \ge k \ge 2p+b-1$. Also let $\Sigma_p^b$ be a compact, connected, oriented genus $p$ surface with $b$ boundary components and $l=2p+b-1$. We define $D$, $\partial^{-}D$, $\partial^{0}D$, and $\partial^{+}D$ as follows:
\[D=\left\{(r,\theta) \ | \ r \in [0,1],\ \theta \in [-\frac{\pi}{3},\frac{\pi}{3}]\right\}, \ \partial^{-}D=\left\{(r,\theta) \ | \ r \in [0,1],\ \theta=-\frac{\pi}{3}\right\},\]
\[\partial^{0}D=\left\{(r,\theta) \ | \ r =1,\ \theta \in [-\frac{\pi}{3},\frac{\pi}{3}]\right\},\ \partial^{+}D=\left\{(r,\theta) \ | \ r \in [0,1],\ \theta=\frac{\pi}{3}\right\}.\]

Then, $\partial{D}=\partial^{-}D \cup \partial^{0}D \cup \partial^{+}D$ holds. We write $P$ for $\Sigma_p^b$ and $U$ for $D\times P$. Then, from the decomposition of $\partial{D}$, we have $\partial{U}=\partial^{-}{U} \cup \partial^{0}{U} \cup \partial^{+}{U}$, where 
\[\partial^{\pm}{U}=\partial^{\pm}{D} \times P,\ \partial^{0}{U}=P \times \partial^{0}{D} \cup \partial{P} \times D.\]

For an integer $n > 0$, let $V_n=\natural_{n}S^1 \times D^3$ and $\partial{V_{n}}=\partial^{-}{V_{n}} \cup \partial^{+}{V_{n}}$ be the standard genus $n$ Heegaard splitting of $\partial{V_{n}}$. Moreover, for an integer $s \ge n$, the Heegaard splitting of $\partial{V_{n}}$ obtained by stabilizing the standard Heegaard splitting is denoted by $\partial{V_{n}}=\partial_{s}^{-}{V_{n}} \cup \partial_{s}^{+}{V_{n}}$. Henceforth, let $n=k-2p-b+1=k-l$, $s=g-k+p+b-1$ ($V_{n}=V_{k-2p-b+1}=V_{k-l}$).

Lastly, we define $Z_k=U \natural V_n$, where the boundary sum is taken by identifying the neighborhood of a point in int($\partial^{-}U \cap \partial^{+}U$) with the neighborhood of a point in int($\partial_{s}^{-}V_{n} \cap \partial_{s}^{+}V_{n}$). Here, we define $Y_k=\partial{Z_k}=\partial{U} \# \partial{V_n}$. Then, from the above decomposition, we have $Y_{k}=Y_{g,k;p,b}^{-} \cup Y_{g,k;p,b}^{0} \cup Y_{g,k;p,b}^{+}$, where $Y_{g,k;p,b}^{\pm}=\partial^{\pm}U \natural \partial_{s}^{\pm}V_n$ and $Y_{g,k;p,b}^{0}=\partial^{0}U=P \times \partial^{0}{D} \cup \partial{P} \times D$.

Using these notations, we can define a relative trisection as follows.

\begin{dfn}\label{def: relative trisection}
Let $X$ be a 4-manifold with connected boundary. The decomposition $X=X_1 \cup X_2 \cup X_3$ of $X$ satisfying the following conditions is called a $(g,k;p,b)$-\textit{relative trisection}:
\begin{itemize}
\item For each $i=1,2,3$, there exists a diffeomorphism $\phi_i \colon X_i \to Z_k$.
\item For each $i=1,2,3$, $\phi_i(X_{i} \cap X_{i-1})=Y_{g,k;p,b}^{-}$, $\phi_i(X_{i} \cap X_{i+1})=Y_{g,k;p,b}^{+}$ and $\phi_i(X_{i} \cap \partial{X})=Y_{g,k;p,b}^{0}$, where $X_{4}=X_{1}$ and $X_{0}=X_{3}$.
\end{itemize} 
\end{dfn}

Note that this definition is that of a \textit{balanced} relative trisection. As with the definition \ref{def:trisection}, we can define an \textit{unbalanced} relative trisection. Moreover in Definition \ref{def: relative trisection}, $X_i \cap X_j \cap \partial{X} \cong \Sigma_{p}^{b}$ must be connected since $\partial{X}$ is assumed to be connected. This fact is used in Section \ref{sec:boundary-stabilization} to consider a relative trisection of the complement of a surface-knot.

Given a relative trisection, we can define a relative trisection diagram.

\begin{dfn}
A $(g,k;p,b)$-\textit{relative trisection diagram} is a 4-tuple $(\Sigma_g^b,\alpha,\beta,\gamma)$ satysfying the following conditions:
\begin{itemize}
\item $\alpha$, $\beta$ and $\gamma$ are respectively $(g-p)$-tuples of curves on $\Sigma_g^b$.
\item Each of the 3-tuples $(\Sigma_{g}^{b},\alpha,\beta)$, $(\Sigma_{g}^{b},\beta,\gamma)$, $(\Sigma_{g}^{b},\gamma, \alpha)$ is diffeomorphism and handleslide equivalent to the diagram described in Figure \ref{fig:model diagram of relative trisection}. 
\end{itemize}
\end{dfn}

\begin{figure}[h]
\centering
\includegraphics[width=11cm, height=7cm, keepaspectratio, scale=1]{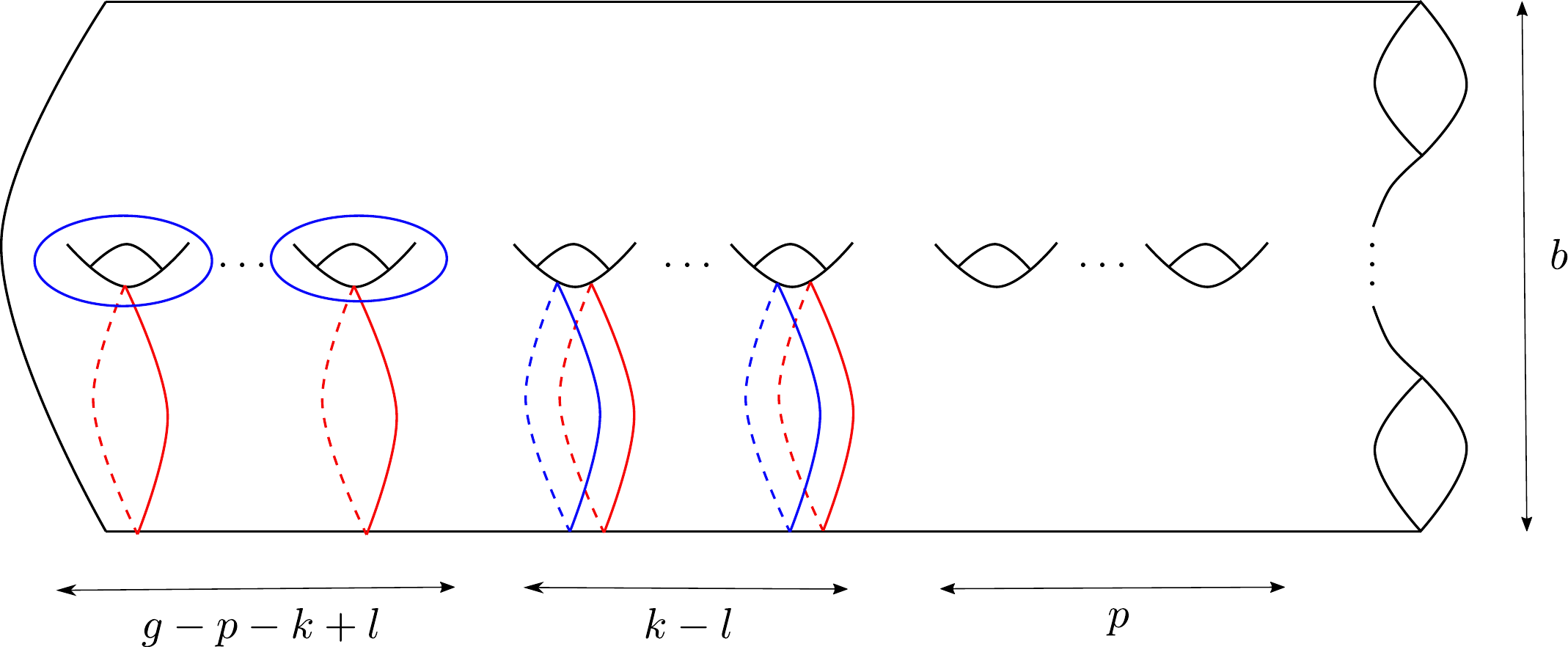}
\setlength{\captionmargin}{50pt}
\caption{The standard diagram for a relative trisection diagram. Note that $l=2p+b-1$.}
\label{fig:model diagram of relative trisection}
\end{figure}

\begin{lem}[Lemma 11 in \cite{CGPC}]\label{lem:open book decomposition}
A $(g,k;p,b)$-relative trisection of a 4-manifold $X$ with non-empty boundary induces an open book decomposition on $\partial{X}$ with page $\Sigma_p^b$ (hence binding $\partial{\Sigma_p^b}$).
\end{lem}

If we want to glue several relative trisection diagrams, we must describe a diagram with arcs, called an \textit{arced relative trisection diagram}. There exists an algorithm for drawing such arcs. 


\begin{lem}[Lemma 2.7 in \cite{CO}]\label{lem:gluing}
For $i=1,2$, let $X_i$ be a 4-manifold with nomempty and connected boundary, and $T_i$ a relative trisection of $X_i$. Also let $\mathcal{O}X_i$ be the open book decomposition on $\partial{X_i}$ induced by $T_i$. If $f \colon \partial{X_1} \to \partial{X_2}$ is an orientation reversing diffeomorphism which takes $\mathcal{O}X_1$ to $\mathcal{O}X_2$, then we obtain a trisection of $X=X_1 \cup_{f} X_2$ by gluing $T_1$ and $T_2$.
\end{lem}

Note that if there exists a diffeomorphism $f$ as above, the page of $\mathcal{O}X_1$ is diffeomorphic to the page of $\mathcal{O}X_2$ via $f$. Thus, if $T_i$ is the $(g_i,k_i;p_i,b_i)$-relative trisection, then $p_1=p_2$ and $b_1=b_2$.

Let $(\Sigma(i),\alpha(i),\beta(i),\gamma(i),a(i),b(i),c(i))$ be an arced relative trisection diagram of $X_i$. If there exsits $f$ in Lem \ref{lem:gluing}, we can obtain three kinds of new simple closed curves in $\Sigma(1) \cup_{f} \Sigma(2)$, i.e. $a(1) \cup a(2)$, $b(1) \cup b(2)$ and $c(1) \cup c(2)$ via $f$. Thus, we have the following proposition, where $\Sigma=\Sigma(1) \cup_{f} \Sigma(2)$ and $\tilde{\alpha}$ (resp. $\tilde{\beta}$, resp. $\tilde{\gamma}$) $=(a(1)_j \cup_{\partial} a(2)_j)_j$ (resp. $(b(1)_j \cup_{\partial} b(2)_j)_j$, resp. $(c(1)_j \cup_{\partial} c(2)_j)_j$. 

\begin{prop}[Proposition 2.12 in \cite{CO}]
In addition to the assumptions in Lem \ref{lem:gluing}, let $(\Sigma(i),\alpha(i),\beta(i),\gamma(i),a(i),b(i),c(i))$ be an arced relative trisection diagram of $X_i$. Then, the 4-tuple $(\Sigma, \alpha, \beta, \gamma)$ is a trisection diagram of $X$, where $\alpha=\alpha(1) \cup \alpha(2) \cup \tilde{\alpha}$.
\end{prop}

\begin{prop}[Theorem 5 in \cite{CGPC}]
Let $(\Sigma,\alpha,\beta,\gamma)$ be a $(g,k;p,b)$ relative trisection diagram and $\Sigma_{\alpha}$ the surface obtained by performing the surgery along $\alpha$. Suppose that this operation comes with an embedding $\phi_{\alpha} \colon \Sigma-\alpha \to \Sigma_{\alpha}$. Consider the following step. 
\begin{enumerate}
\item Choose a collection of arcs $a$ such that $a$ is disjoint from $\alpha$ in $\Sigma$ and $\phi_{\alpha}(a)$ cuts $\Sigma_{\alpha}$ into a disk. Note that $a$ consists of $2p+b-1$ arcs. 
\item Choose $b$ by handle sliding $a$ over $\alpha$ so that $b$ is disjoint from $\beta$. If necessary, we slide $\beta_i$ over $\beta_j$. In this case, the $\beta$ is denoted by $\beta{'}$. If handle slides are not needed, $\beta{'}=\beta$.
\item Choose $c$ by handle sliding $b$ over $\beta{'}$ so that $c$ is disjoint from $\gamma$. If necessary, we slide $\gamma_i$ over $\gamma_j$. In this case, the $\gamma$ is denoted by $\gamma{'}$. If handle slides are not needed, $\gamma{'}=\gamma$.
\end{enumerate}
Then, $(\Sigma, \alpha, \beta{'}, \gamma{'}, a, b, c)$ is an arced relative trisection diagram.
\end{prop}

\begin{exam}
Figure \ref{fig:D^2 bundle over S^2} is a $(2,1;0,2)$-relative trisection diagram of the $D^2$ bundle over $S^2$ with Euler number $-1$ and its arced relative trisection diagram constructed from the algorithm.
\end{exam}

\begin{figure}[h]
\centering
\includegraphics[width=8.5cm, height=6.5cm, keepaspectratio, scale=1]{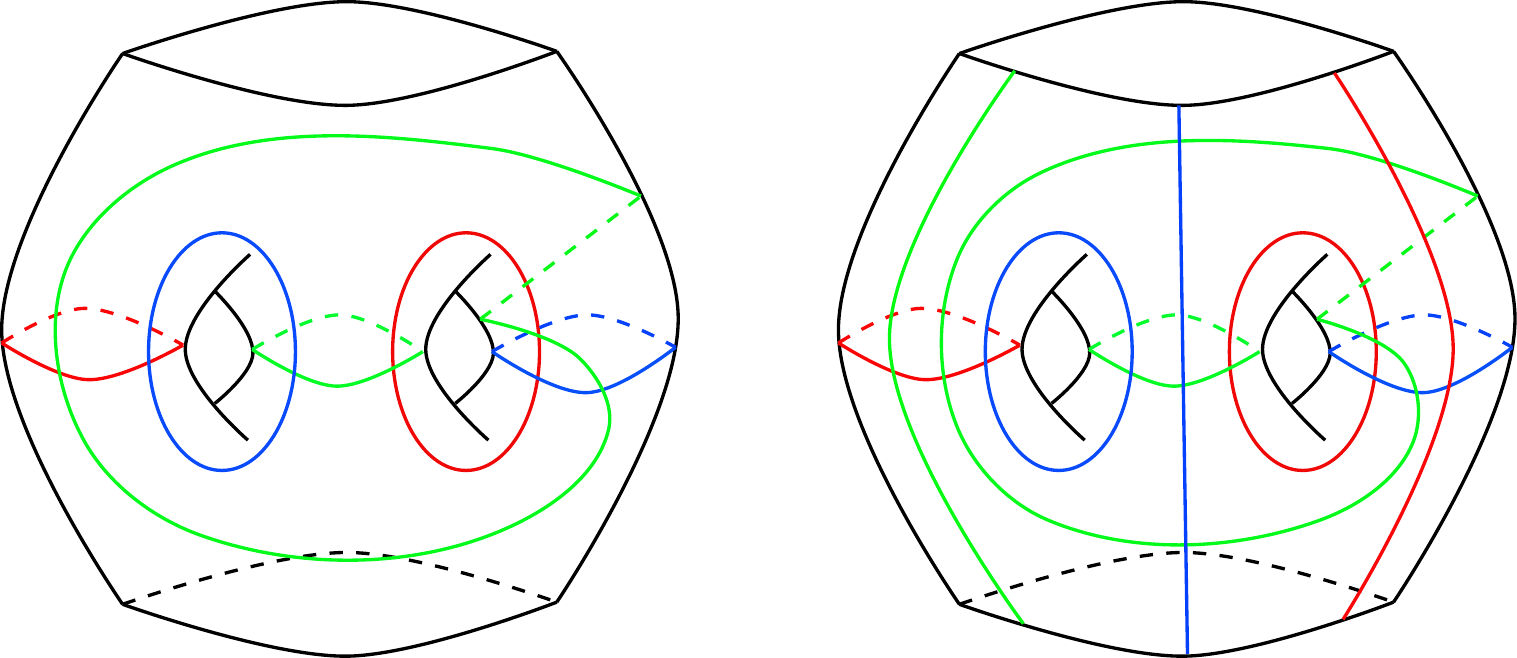}
\setlength{\captionmargin}{50pt}
\caption{(Left) A $(2,1;0,2)$-relative trisection diagram of the $D^2$ bundle over $S^2$ with Euler number $-1$. (Right) Its arced relative trisection diagram.}
\label{fig:D^2 bundle over S^2}
\end{figure}

For more details on relative trisections, see \cite{Ca, CGPC, CO}.

\subsection{Bridge trisections}

In this subsection, we review trisections of surface-knots, called bridge trisections. 

\begin{dfn}
Let $V$ be a 4-dimensional 1-handlebody and $\mathcal{D}$ a collection of disks properly embedded in $V$. We say that $\mathcal{D}$ is \textit{trivial} if the disks of $\mathcal{D}$ are simultaneously isotoped into $\partial{V}$.
\end{dfn}

\begin{dfn}\label{def:trivial tangle}
Let $H$ be a 3-dimensional 1-handlebody and $\tau=\{\tau_i\}$ a collection of arcs properly embedded in $H$. We say that $\tau$ is \textit{trivial} if $\tau_i$ is isotoped into $\partial{H}$ for each $i$. Or equivalently, there exists a collection $\Delta=\{\Delta_i\}$ of disks in $H$ with $\Delta_i \cap \Delta_j =\emptyset$ such that $\partial{\Delta_i} = \tau_i \cup \tau_i^{'}$ for some arc $\tau_i^{'} \subset \partial{H}$. We call $\tau$, $\Delta$ and $\tau_i^{'}$ \textit{trivial tangles}, \textit{bridge disks} and a \textit{shadow} of $\tau_i$ respectively.
\end{dfn}

\begin{dfn}[\cite{MZ2}]\label{def:bridge trisection}
Let $X=X_1 \cup X_2 \cup X_3$ be a $(g;k_1,k_2,k_3)$-trisection of a closed 4-manifold $X$, and $S$ a surface-knot in $X$. A decomposition $(X,S) = (X_1,\mathcal{D}_1) \cup (X_2,\mathcal{D}_2) \cup (X_3,\mathcal{D}_3)$ is a $(g;k_1,k_2,k_3;b;c_1,c_2,c_3)$-\textit{bridge trisection} of $(X,S)$ if
\begin{itemize}
\item For each $i=1,2,3$, $\mathcal{D}_i$ is a collection of trivial $c_i$ disks in $X_i$.
\item For $i \not= j$, $\mathcal{D}_i \cap \mathcal{D}_j$ form trivial $b$ tangles in $X_i \cap X_j$.
\end{itemize}
We say that $S$ is in $(b;c_1,c_2,c_3)$-\textit{bridge position} with respect to $(X_1,X_2,X_3)$ if $(X,S) = (X_1,S \cap X_1) \cup (X_2,S \cap X_2) \cup (X_3,S \cap X_3)$ is a $(g;k_1,k_2,k_3;b;c_1,c_2,c_3)$-bridge trisection.
\end{dfn}

We call the trisection $(X_1,X_2,X_3)$ the \textit{underlying trisection} of the bridge trisection.

\begin{rem}\label{rem:bridge trisection in $S^4$}
In Definition \ref{def:bridge trisection}, if $X=S^4$, then the trisection is the $(0,0)$-trisection \cite[Definition 1.2]{MZ1}.
\end{rem}

As with a balanced trisection, when $k_1=k_2=k_3=k$ and $c_1=c_2=c_3=c$, we say that the decomposition of $(X,S)$ is a $(g,k;b,c)$-\textit{bridge trisection} and $S$ is in $(b,c)$-\textit{bridge position}. Note that if $S$ is in $(b;c_1,c_2,c_3)$-bridge position, then $\chi(S)=c_1+c_2+c_3-b$. So, when $c_1=c_2=c_3$, we often say that $S$ is in $b$-bridge position.

Meier and Zupan \cite{MZ2} showed that every pair of a 4-manifold $X$ and a surface-knot $S$ in $X$ admits a bridge trisection, using a technical operation called \textit{meridional stabilization}. 

\begin{dfn}
Let $(X,S)=(X_1, \mathcal{D}_1) \cup (X_2, \mathcal{D}_2) \cup (X_3, \mathcal{D}_3)$ be a bridge trisection and $C$ an arc in $\mathcal{D}_i \cap \mathcal{D}_j$ whose endpoints are in distinct components of $\mathcal{D}_k$. We define $(X_\ell^{'},\mathcal{D}_\ell^{'})$ as follows, where $\{i,j,k\}=\{1,2,3\}$.
\begin{itemize}
\item $(X_i^{'},\mathcal{D}_i^{'})=(X_i - \nu(C), \mathcal{D}_i - \nu(C))$ 
\item $(X_j^{'},\mathcal{D}_j^{'})=(X_j - \nu(C), \mathcal{D}_j - \nu(C))$ 
\item $(X_k^{'},\mathcal{D}_k^{'})=(X_k \cup \overline{\nu(C)}, \mathcal{D}_k \cup (\overline{\nu(C)} \cap S)$
\end{itemize}
The replacement of $(X_\ell,\mathcal{D}_\ell)$ by $(X_\ell^{'},\mathcal{D}_\ell^{'})$ for all $\ell$ is said to be a $k$-meridionally stabilization.
\end{dfn}

Note that when we meridionally stabilize a bridge trisection of $(X,S)$, for the underlying trisection of $X$, we simply stabilize it. This observation is used in the proof of our main theorem. 

\begin{thm}[Theorem 2 in \cite{MZ2}]\label{thm:good bridge position}
Let $S$ be a surface-link in a closed 4-manifold $X$ with a $(g,k)$-trisection $T$. Then, the pair $(X,S)$ admits a $(g,k;b,n)$-bridge trisection with $b=3n-\chi(S)$, where $n$ is the number of connected components of $S$. 
\end{thm}

Note that in Theorem \ref{thm:good bridge position}, if $S$ is a 2-knot, then $S$ can be in 1-bridge position with respect to a trisection obtained by stabilizing $T$. Furthermore if $S$ is a $P^2$-knot, then $S$ can be in 2-bridge position.


A surface-knot in $S^4$ can be described by a triplane diagram introduced by Meier and Zupan \cite{MZ1}. On the other hand, it is difficult to describe a surface-knot in a general 4-manifold in the same way. Therefore, Meier and Zupan \cite{MZ2} developed another diagram using shadows in Definition \ref{def:trivial tangle}. It is called a shadow diagram. 

\begin{dfn}
Let $(X,S)=(X_1,\mathcal{D}_1) \cup (X_2,\mathcal{D}_2) \cup (X_3,\mathcal{D}_3)$ be a bridge trisection. A 4-tuple $(\Sigma, (\alpha, a), (\beta, b), (\gamma, c))$ is called a \textit{shadow diagram} if the 4-tuple $(\Sigma,\alpha,\beta,\gamma)$ is a trisection diagram of $(X_1,X_2,X_3)$, and $a$, $b$ and $c$ are shadows of $\mathcal{D}_1 \cap \mathcal{D}_2$, $\mathcal{D}_2 \cap \mathcal{D}_3$ and $\mathcal{D}_3 \cap \mathcal{D}_1$ respectively.  In particular, $a$, $b$ and $c$ are a shadow of $\mathcal{D}_i \cap \mathcal{D}_j$, the shadow diagram is called a \textit{doubly pointed trisection diagram}.
\end{dfn}

Each 2-knot in a close 4-manifold admits a doubly pointed trisection diagram since it can be put in 1-bridge position. Note that for a 2-knot $K$ in 1-bridge position with respect to a trisection $T$ of $X$, the underlying trisection diagram of $(X,K)$ is the diagram of $T$. For example, Figure \ref{fig:dptd of CP^2 and CP^1} describes a doubly pointed trisection diagram of $(\mathbb{C}P^2, \mathbb{C}P^1)$. We call the two black points of a doubly pointed trisection diagram \textit{base points} in the proof of our main theorem.

For more details on bridge trisections, see \cite{MZ1, MZ2}.

\begin{figure}[h]
\begin{center}
\includegraphics[width=5.5cm, height=6cm, keepaspectratio, scale=1]{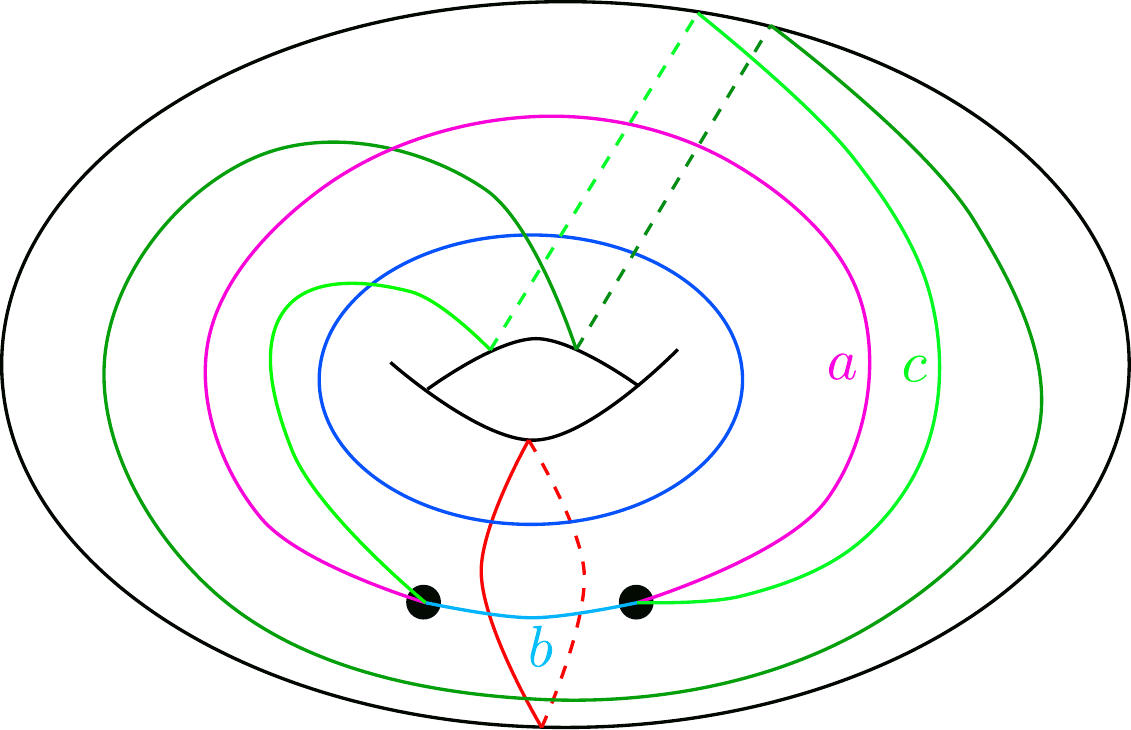}
\end{center}
\setlength{\captionmargin}{50pt}
\caption{A doubly pointed trisection diagram of $(\mathbb{C}P^2,\mathbb{C}P^1)$. The red, blue, and green curves describe a $(1,1)$-trisection diagram of $\mathbb{C}P^2$ and the arcs $a$, $b$, and $c$ describe $\mathbb{C}P^1$. Note that in a doubly pointed trisection diagram, we do not need to draw the arcs since there is a unique way to describe them.}
\label{fig:dptd of CP^2 and CP^1}
\end{figure}

\section{The Price twist}\label{sec:Price twist}

In this section, we review a surgery along a $P^2$-knot in a closed 4-manifold, called the Price twist.

Let $S$ be a $P^2$-knot, that is, a real projective plane smoothly embedded in a closed 4-manifold $X$, with normal Euler number $e(S)=\pm2$. Note that when $X=S^4$, from Whitney-Massey's theorem \cite{Ma, Wh}, each $P^2$-knot $S$ satisfies $e(S)=\pm2$. Then, for a tubular neighborhood $\nu(S)$ of $S$ in $X$, the boundary $\partial{\nu(S)}$ is a Seifert-fibered space $Q$ over $S^2$ with three singular fibers labeled $S_0$, $S_1$ and $S_{-1}$, where these indices are respectively $\pm2$, $\pm2$ and $\mp2$ when $e(S)=\pm2$. Since $\partial(X-\nu(S)) \cong Q$, $\partial(X-\nu(S))$ has the same label with $\partial{\nu(S)}$. Price \cite{Pr} showed that there exist three kinds of  self-homeomorphism of  $\partial{\nu(S)}$ up to isotopy, that is, $S_{-1} \mapsto S_{-1}$, $S_{-1} \mapsto S_{0}$ and $S_{-1} \mapsto S_{1}$. Thus, when we reglue $\nu(S)$ deleted from $X$ according to $\phi \colon \partial \nu(S) \to \partial(X-\nu(S))$, we can obtain the following at most (see below) three 4-manifolds up to diffeomorphism (the notation follows \cite{KM}):
\begin{itemize}
\item If $\phi(S_{-1})=S_{-1}$, the resulting manifold is $X$.
\item If $\phi(S_{-1})=S_{0}$, the resulting manifold is denoted by $\tau_S(X)$. 
\item If $\phi(S_{-1})=S_{1}$, the resulting manifold is denoted by $\Sigma_S(X)$.
\end{itemize}

This operation is called the \textit{Price twist} of $X$ along $S$. Especially, in this paper, we call the first twist, that is, the twist having the original manifold $X$, the \textit{trivial Price twist}. Note that $\Sigma_S(S^4)$ is a homotopy 4-sphere. Let $\Sigma_K^G(X)$ be the 4-manifold obtained by the Gluck twist along $K$, where $K$ is a 2-knot in $X$. Then, from \cite{KSTY}, we see that for a $P^2$-knot $S=K\#P_{\pm}$, $\Sigma_S(X) \cong \Sigma_K^G(X)$ holds, where $P_{\pm}$ is an unknotted $P^2$-knot with normal Euler number $\pm2$ in $X$. So, for a 2-knot $K$ satisfying $\Sigma_K^G(S^4) \cong S^4$ such as a twist spun 2-knot,  we have $\Sigma_S(S^4) \cong S^4$. Thus, we can ask whether the conjecture that is a 4-dimensional analogue of Waldhausen's theorem on Heegaard splittings \cite[Conjecture 3.11]{MSZ} holds for such $\Sigma_S(S^4)$ (\cite[Question 6.2]{KM}). Note that \cite[Question 6.2]{KM} is a specific case of \cite[Conjecture 3.11]{MSZ}.

\begin{que*}[Question 6.2 in \cite{KM}]\label{KM's Question}
Let $S$ be a $P^2$-knot in $S^4$ so that $\Sigma_S(S^4) \cong S^4$. Is a trisection of  $\Sigma_S(S^4)$ obtained from the algorithm of Section 5 in \cite{KM} isotopic to a stabilization of the genus 0 trisection of $S^4$?
\end{que*}

\begin{con*}[Conjecture 3.11 in \cite{MSZ}]
Every trisection of $S^4$ is isotopic to either the genus 0 trisection or its stabilization.
\end{con*}

\begin{rem}
From \cite{Na}, we immediately see that Figure 19 right of \cite{KM}, that is, a $(6,2)$-trisection diagram of $\Sigma_{P_-}(S^4)$, is a stabilization of the $(0,0)$-trisection diagram of $S^4$ up to handle slides and diffeomorphisms. 
\end{rem}

\begin{rem}\label{rem:Kisoshita conjecture}
For a 2-knot $K$ and an unknotted $P^2$-knot $P$ in $S^4$, the $P^2$-knot $S$ admits the decomposition $K\#P$ is said to be \textit{of Kinoshita type}. It is not known whether every $P^2$-knot in $S^4$ is of Kinoshita type. This question is called the Kinoshota question or the Kinoshita conjecture. We may answer the question with $\tau_S({S^4})$ \cite{KM}. Note that in \cite[Question 6.2]{KM}, if $S$ is of Kinoshita type, then trisections in the question are diffeomorphic to trisections obtained by the Gluck twist \cite{Na}. In particular, if $S$ is the connected sum of the unknotted $P^2$-knot and a spun or twist spun 2-knot, \cite[Question 6.2]{KM} reduces to \cite[Question 6.4]{GM} in the sense of diffeomorphic trisections.
\end{rem}

\begin{que*}[Question 6.4 in \cite{GM}]
Is the trisection diagram constructed by  \cite{Me} and \cite[Lemma 5.5]{GM} for the Gluck twist along a spun or twist spun 2-knot is a stabilization of the $(0,0)$-trisection diagram of $S^4$?
\end{que*}

This question is not answered even in the case of the spun trefoil, which can be regarded as the simplest non trivial spun 2-knot.

By the following theorem, called Waldhausen's theorem, we can see the reason that \cite[Conjecture 3.11]{MSZ} is a 4-dimensional analogue of Waldhausen's theorem on Heegaard splittings.

\begin{thm}[\cite{Wal},\cite{Sch}]
The 3-sphere $S^3$ admits a unique Heegaard splitting up to isotopy for each genus.
\end{thm}

For more details on the Price twist and a  trisection obtained by the Price twist, see \cite{KM, Pr}.

\section{A boundary-stabilization}\label{sec:boundary-stabilization}

In this section, we review a boundary-stabilization for a 4-manifold with boundary introduced in \cite{KM}. 

\begin{dfn}
Let $Y=Y_1 \cup Y_2 \cup Y_3$ be a 4-manifold with $\partial{Y}\not=\emptyset$, where $Y_i \cap Y_j = \partial{Y_i} \cap \partial{Y_j}$, and $C$ an arc properly embedded in $Y_i \cap Y_j \cap \partial{Y}$ whose endpoints are in $Y_1 \cap Y_2 \cap Y_3$. Also let $\nu(C)$ be a fixed open tubular neighborhood of $C$. Then, we define $\tilde{Y_i}, \tilde{Y_j}, \tilde{Y_k}$ as follows:
\begin{itemize}
\item $\tilde{Y_i} = Y_i - \nu(C)$,
\item $\tilde{Y_j} = Y_j - \nu(C)$,
\item $\tilde{Y_k} = Y_k \cup \overline{\nu(C)}$.
\end{itemize} 
The replacement of $(Y_1,Y_2,Y_3)$ by $(\tilde{Y_i}, \tilde{Y_j}, \tilde{Y_k})$ is said to be a \textit{boundary-stabilization} along $C$. In this case, we say that $\tilde{Y_k}$ has been obtained by boundary-stabilizing $Y_k$ along $C$.
\end{dfn}

As we have seen in Section \ref{sec:intro}, we need a boundary-stabilization in order to construct a relative trisection of the complement of a surface-knot in a closed 4-manifold. The following explanation is more precise.

Let $S$ be a surface-knot in a closed 4-manifold $X$ with trisection $(X_1,X_2,X_3)$. Suppose that $S$ is in $(b,c)$-bridge position with respect to $(X_1,X_2,X_3)$. Let $X_i^{'} = X_i-\nu(S)$. Then, $X-\nu(S)$ admits a natural decomposition $X-\nu(S) = X_1^{'} \cup X_2^{'} \cup X_3^{'}$. However, this decomposition of $X-\nu(S)$ admits $(X_1^{'}, X_2^{'}, X_3^{'})$ as a relative trisection if and only if $S$ is a 2-knot and $S$ is in 1-bridge position, that is, $b=1$. This is because if $b>1$, then the triple intersection $X_i^{'} \cap X_j^{'} \cap \partial{(X-\nu(S))}$, which is diffeomorphic to the disjoint union $\sqcup_{b} S^1 \times I$ of $b$ annuli, is disconnected. This contradicts the fact that for a relative trisection $(Y_1,Y_2,Y_3)$, if $\partial{Y}$ is connected, then $Y_i \cap Y_j \cap \partial{Y}$ must be connected.
So, for all $S$ except 2-knots, $X-\nu(S)$ cannot admit $(X_1^{'}, X_2^{'}, X_3^{'})$ as a relative trisection. Although, we can refine the decomposition by boundary-stabilizing each $X_i^{'}$ so that $X-\nu(S)$ admits a relative trisection for each $S$. Put briefly, the way is the following:

In this paper, since we focus on a $P^2$-knot, we first review a boundary-stabilization of the complement of a $P^2$-knot. In the above situation, suppose also that $S$ is a $P^2$-knot and $b=2$ (Theorem \ref{thm:good bridge position}). For each $i=1,2,3$ and $\{i,j,k\} = \{1,2,3\}$, we define $C_i$ to be an arc in $X_j^{'} \cap X_k^{'} \cap \partial(X-\nu(S))$ whose endpoints are in $X_1^{'} \cap X_2^{'} \cap X_3^{'}$ which intersects two distinct connected components of $\partial({X_1^{'} \cap X_2^{'} \cap X_3^{'}})$. Take $C_1$, $C_2$ and $C_3$ so that they have different endpoints. Then, if we boundary-stabilize $X_\ell^{'}$ along $C_\ell$, we obtain the decomposition $X-\nu(S) = \tilde{X_1} \cup \tilde{X_2} \cup \tilde{X_3}$, where $\tilde{X_\ell}$ is the submanifold of $X-\nu(S)$ obtained by boundary-stabilizing $X_\ell^{'}$ along $C_1$, $C_2$, and $C_3$. We see that $\tilde{X_i} \cap \tilde{X_j} \cap \partial(X-\nu(S))$ is connected and if we furthermore check on the structure of an open book decomposition which will be induced, we have the following proposition.

\begin{prop}[\cite{KM}]
The 3-tuple $(\tilde{X_1}, \tilde{X_2}, \tilde{X_3})$ is a relative trisection of $X-\nu(S)$.
\end{prop}

For a surface-knot $S$ except $P^2$-knots, we can construct a relative trisection of the complement of $S$ as with the case of a $P^2$-knot. The differencies are that for each $i=1,2,3$, we take $C_i$ to be a collection of $2-\chi(S)$ arcs and take each arc in $C_i$ so that the arc is parallel to a different one in $\nu(S) \cap X_j \cap X_k$.

Note that unlike a stabilization of a trisection, a boundary-stabilization depends on the choice of an arc. If $S$ is a $P^2$-knot, the type of a relative trisection of $X-\nu(S)$ obtained by boundary-stabilizations as above is either $(g,k;0,3)$ or $(g^{'},k^{'};1,1)$. In Section \ref{sec:main theorem}, since we glue a $(2,2;0,3)$-relative trisection of $\overline{\nu(S)}$ and a relative trisection of $X-\nu(S)$ from boundary-stabilizations, we need to boundary-stabilize $X-\nu(S)=\bigcup_{i=1}^{3} X_{i} - \nu(S)$ so that the type of the resulting relative trisection is $(g,k;0,3)$ for some $g$ and $k$.

Kim and Miller developed an algorithm to describe a relative trisection diagram of the complement of a surface-knot using the shadow diagram; see \cite[Section 4]{KM}.

For more details on boundary-stabilizations and a relative trisection of the complement of a surface-knot, see \cite{KM}.

\section{Main Theorem}\label{sec:main theorem}

As we have seen in Section \ref{sec:intro}, we can think about the following question.

\begin{que}\label{que:original question}
Let $S$ be a surface-knot in a  closed 4-manifold $X$ with trisection $T$. Is a trisection obtained by trivially gluing $\nu(S)$ and $X-\nu(S)$ diffeomorphic, especially isotopic, to a stabilization of $T$? In particular, if $X=S^4$, does this hold?
\end{que}

For the restricting case, we answer Question \ref{que:original question} affirmatively in Theorem \ref{main theorem}, our main theorem.

\begin{thm}\label{main theorem}
Let $X$ be a closed 4-manifold and $S$ the connected sum of a 2-knot $K$ with normal Euler number 0 and an unknotted $P^2$-knot with normal Euler number $\pm2$ in $X$. Also let $T_{(X,S)}$ be a bridge trisection of $(X,S)$ and $T_X$ the underlying trisection. Suppose that $S$ is in bridge position with respect to $T_X$. Also let $T_{X}^{'}$ be the underlying trisection of the bridge trisection obtained by meridionally stabilizing $T_{(X,S)}$ so that $S$ is in 2-bridge position with respect to $T_{X}^{'}$. Then, the trisection $T_S$ obtained by the trivial Price twist along $S$ is diffeomorphic to a stabilization of $T_{X}^{'}$. In particular, the trisection $T_S$ is diffeomorphic to a stabilization of $T_{X}$.
\end{thm}


\begin{proof}[Proof of Theorem \ref{main theorem}]
Let ${\mathcal{D}}_Y$ be a relative trisection diagram of a 4-manifold $Y$. Also let $P_+$ and $P_-$ be unknotted $P^2$-knots in $X$ with normal Euler number $2$ and $-2$, respectively.

Since the preferred diagram ${\mathcal{D}}_{\nu(P_+)}$ and ${\mathcal{D}}_{S^4-\nu(P_+)}$ in \cite{KM} are the mirror images of ${\mathcal{D}}_{\nu(P_-)}$ and ${\mathcal{D}}_{S^4-\nu(P_-)}$, respectively, it suffices to proof Theorem \ref{main theorem} only for $S=K \# P_-$.

\subsection*{Constructing $T_S$}
It follows from \cite{KM} that ${\mathcal{D}}_{X-{\nu(S)}}$ is the union of ${\mathcal{D}}_{S^4-\nu(P_-)}$ and ${\mathcal{D}}_{X-\nu(K)}$. Thus, the gluing  ${\mathcal{D}}_{\nu(P_-)}$ and ${\mathcal{D}}_{X-{\nu(S)}}$ together by the trivial Price twist is described as Figure \ref{fig:gluing diagram}. Note that we construct $\mathcal{D}_{\nu(P_-)}$ in Figure \ref{fig:gluing diagram} by deforming the preferred diagram of $\nu(P_-)$ in \cite{KM} so that the gluing is described as Figure \ref{fig:gluing diagram}. In Figure \ref{fig:gluing diagram}, if we draw arcs of ${\mathcal{D}}_{\nu(P_-)}$ and ${\mathcal{D}}_{S^4-\nu(P_-)}$, then we can obtain Figure \ref{fig:starting diagram}. The diagram depicted in Figure \ref{fig:starting diagram} corresponds to $T_S$. It should be noted that we do not draw curves and arcs on the surface of
${\mathcal{D}}_{X-\nu(K)}$ in Figure \ref{fig:gluing diagram}, but ${\mathcal{D}}_{X-\nu(K)}$ has them.

\begin{figure}[h]
\begin{center}
\includegraphics[width=16cm, height=7.95cm, keepaspectratio, scale=1]{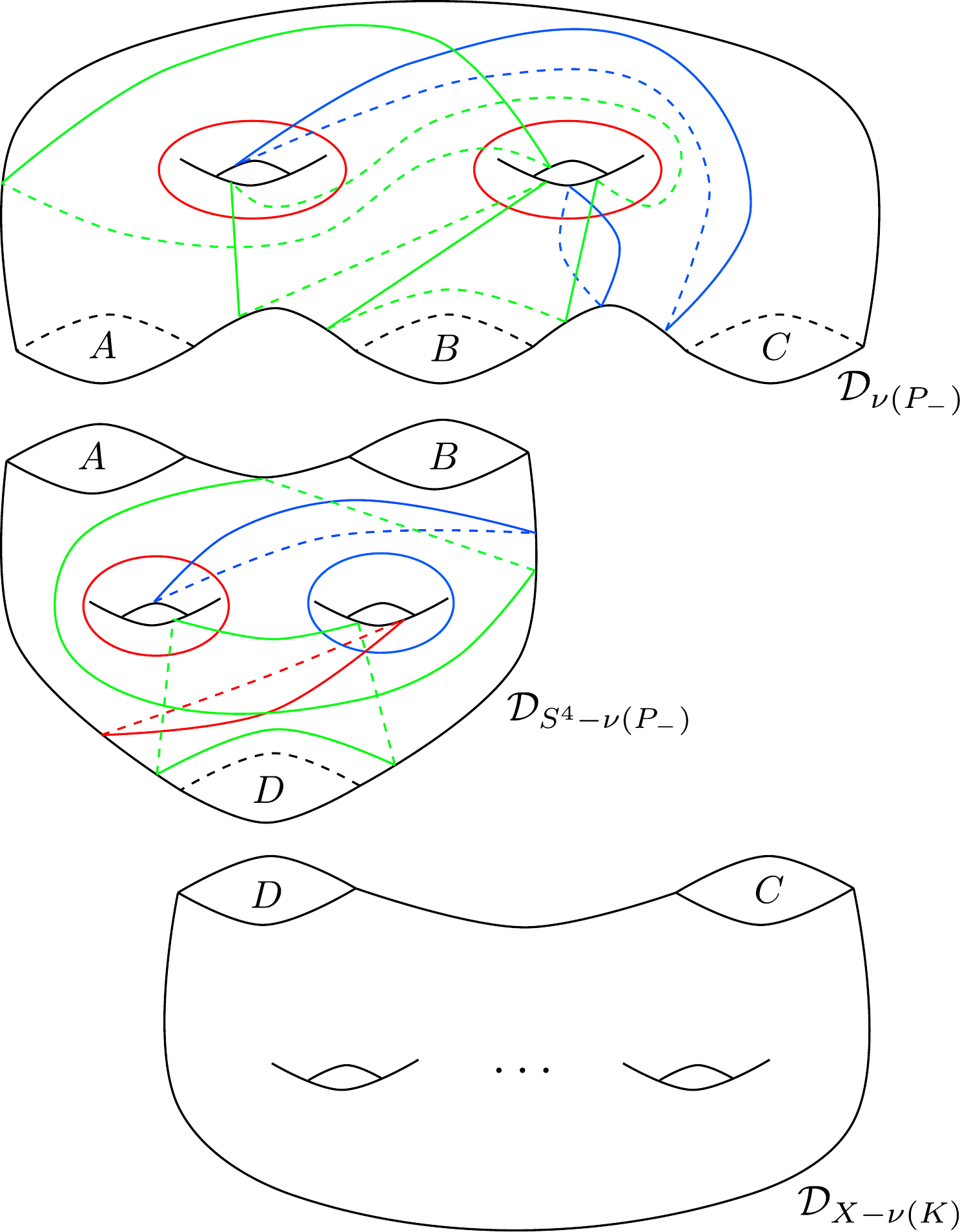}
\end{center}
\setlength{\captionmargin}{50pt}
\caption{The gluing diagram of ${\mathcal{D}}_{\nu(P_-)}$ and ${\mathcal{D}}_{X-{\nu(S)}}$ by the trivial Price twist along $S=K\# P_-$ in $X$. We glue ${\mathcal{D}}_{\nu(P_-)}$, ${\mathcal{D}}_{S^4-\nu(P_-)}$, and ${\mathcal{D}}_{X-\nu(K)}$ along the boundary components of the corresponding characters.}
\label{fig:gluing diagram}
\end{figure}

From now on, we deform trisection diagrams specifically. Note that from Figure \ref{fig:starting diagram} to Figure \ref{fig:after fifth destabilization}, the undrawn part describes ${\mathcal{D}}_{X-\nu(K)}$ with arcs and if necessary, let two arcs of $\mathcal{D}_{X-\nu(K)}$ be parallel by performing handle slides. Also note that for a $\alpha$ curve $\alpha_i$, we call a curve obtained by sliding $\alpha_i$ over another $\alpha$ curve also $\alpha_i$. The same is true for $\beta$ and $\gamma$ curves.

\subsection*{The first destabilization}
In Figure \ref{fig:starting diagram} (or Figure \ref{fig:before first destabilization}), we will destabilize $\alpha_1$, $\beta_1$ and $\gamma_1$. To do this, we slide $\gamma_{2}$ over $\gamma_{3}$ so that the geometric intersection number of $\gamma_{2}$ and $\alpha_{1}$ is 2. Then, we slide $\gamma_{2}$, $\gamma_{3}$ and $\gamma_{4}$ over $\gamma_{1}$ in this order. After that, we slide $\gamma_{2}$ over $\gamma_{4}$. As a result, $\gamma_{2}$ does not intersect $\alpha_1$. We also slide $\gamma_{4}$ over $\gamma_{3}$ so that $\gamma_{4}$ does not intersect $\alpha_1$. Finally, we slide $\gamma_3$ over $\gamma_1$, so that all $\gamma$ curves except $\gamma_1$ do not meet $\alpha_1$ and $\beta_1$. Then, we obtain Figure \ref{fig:before first destabilization}. In Figure \ref{fig:before first destabilization}, by destabilizing $\alpha_1$, $\beta_1$ and $\gamma_1$, that is, erasing $\gamma_1$ and surgering $\alpha_1$ or $\beta_1$ (if we choose $\alpha_1$, then we erase $\beta_1$ and vise versa), we get Figure \ref{fig:after first destabilization}.

\begin{figure}[h]
\begin{center}
\includegraphics[width=16cm, height=7cm, keepaspectratio, scale=1]{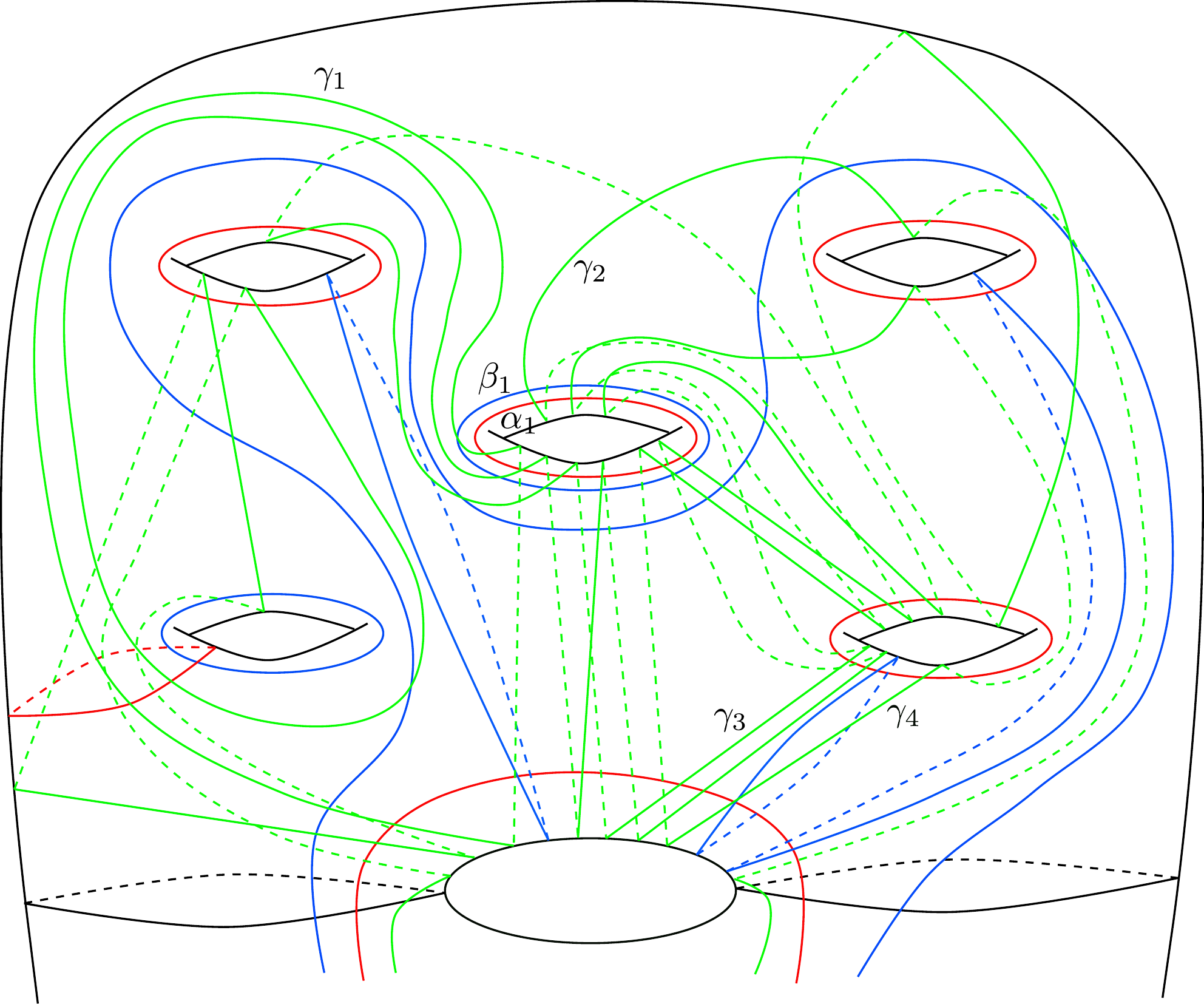}
\end{center}
\caption{Starting diagram.}
\label{fig:starting diagram}
\end{figure}
 
\begin{figure}[h]
\begin{center}
\includegraphics[width=16cm, height=7cm, keepaspectratio, scale=1]{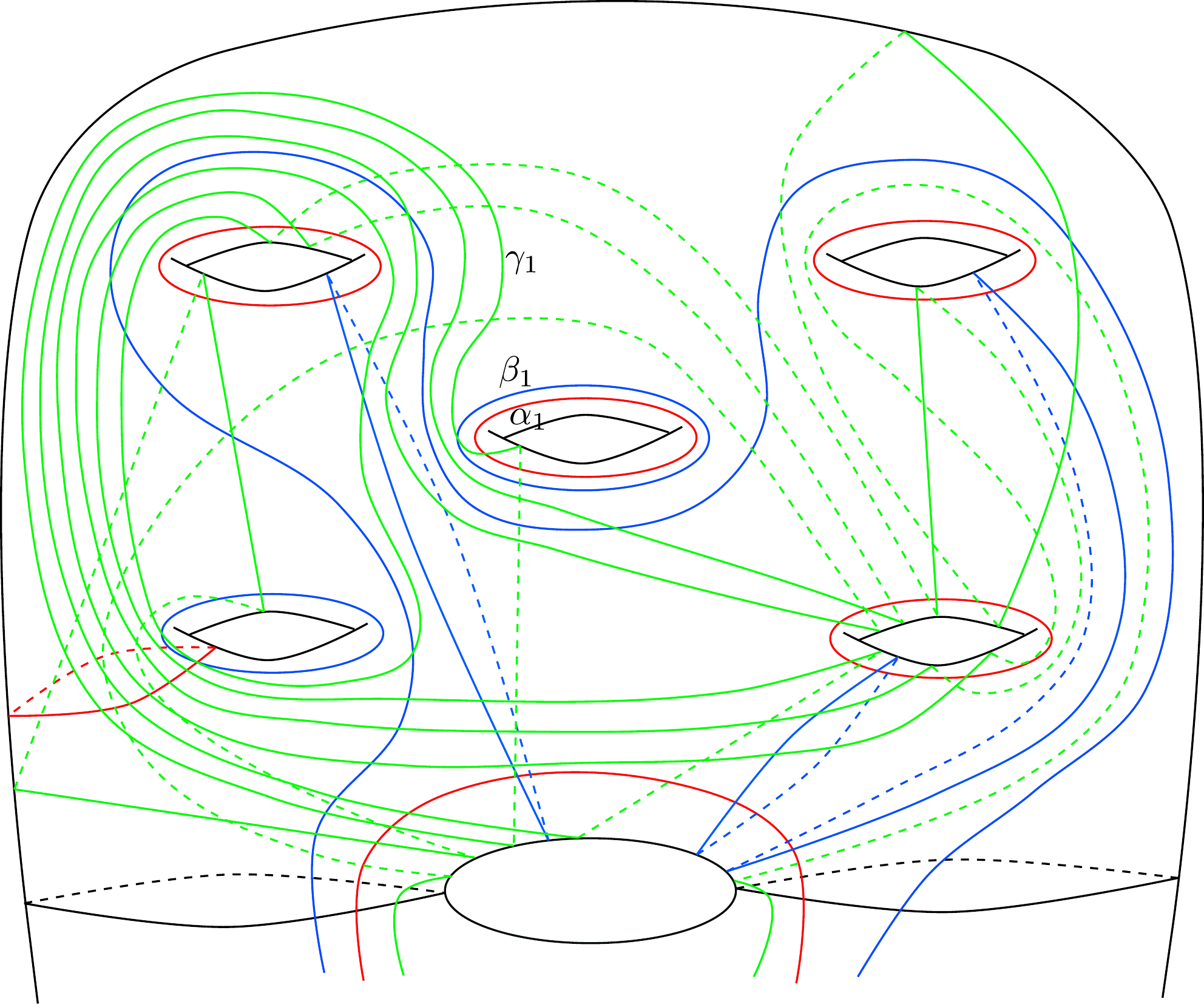}
\end{center}
\caption{Before the first destabilization.}
\label{fig:before first destabilization}
\end{figure}

\begin{figure}[h]
\begin{center}
\includegraphics[width=16cm, height=7cm, keepaspectratio, scale=1]{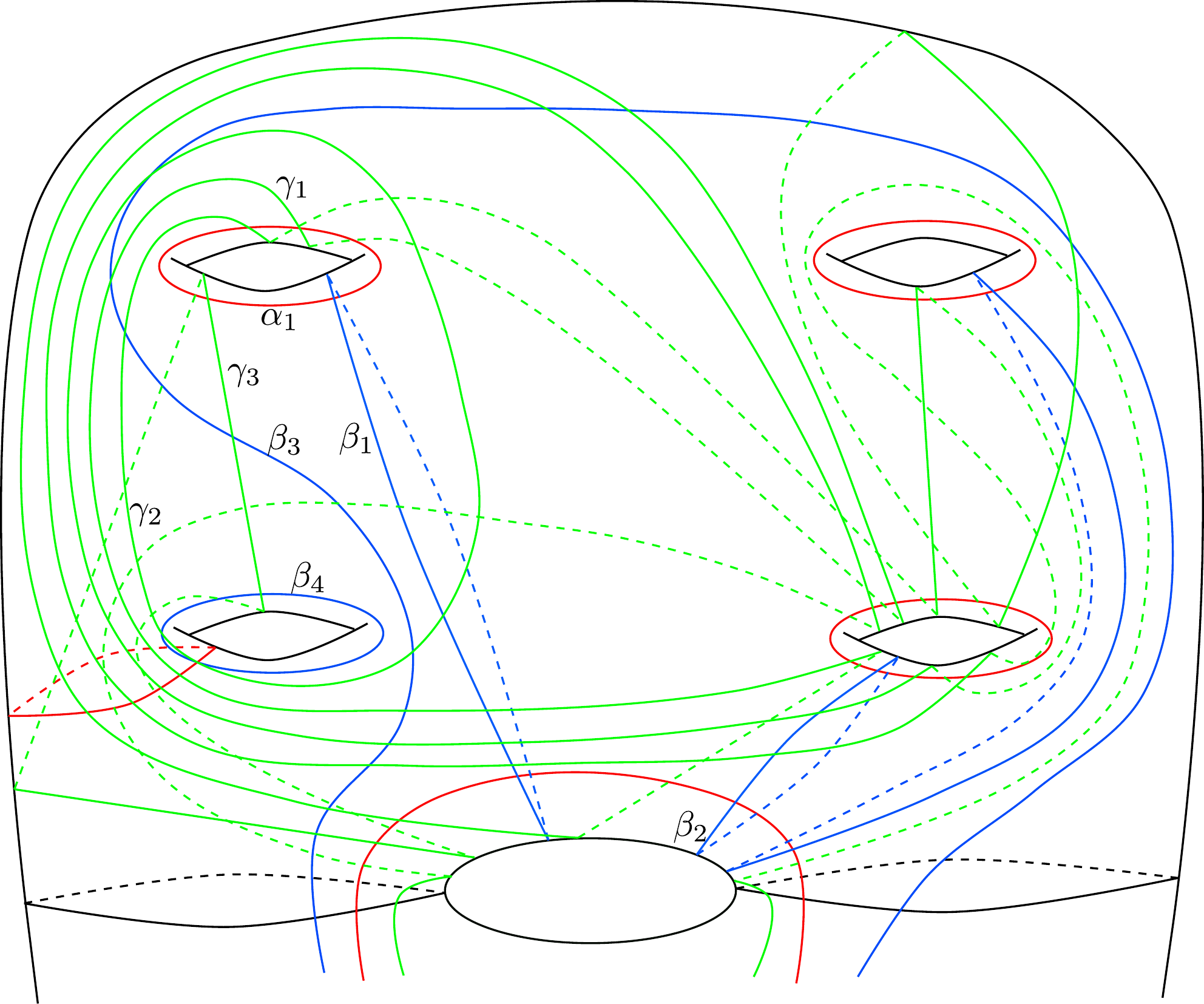}
\end{center}
\caption{After the first destabilization.}
\label{fig:after first destabilization}
\end{figure}

\subsection*{The second destabilization}
In Figure \ref{fig:after first destabilization} (or Figure \ref{fig:before second destabilization}),
we will destabilize $\alpha_1$, $\beta_1$ and $\gamma_1$. To do this, we firstly need to make $\beta_{1}$ parallel to $\gamma_1$. We slide $\beta_3$ over $\beta_4$ and $\beta_1$ over $\beta_2$. We again slide $\beta_1$ over $\beta_2$ so that $\beta_1$ is parallel to $\gamma_1$. After that, we slide $\gamma_2$ and $\gamma_3$ over $\gamma_1$ in order to remove the crossings of $\gamma_2$, $\gamma_3$ and $\alpha_1$. As a result, we obtain Figure \ref{fig:before second destabilization}. In Figure \ref{fig:before second destabilization}, by destabilizing $\alpha_1$, $\beta_1$ and $\gamma_1$, that is, erasing $\beta_1$ and $\gamma_1$ and surgering $\alpha_1$, we get Figure \ref{fig:after second destabilization}.

\begin{figure}[h]
\begin{center}
\includegraphics[width=16cm, height=7cm, keepaspectratio, scale=1]{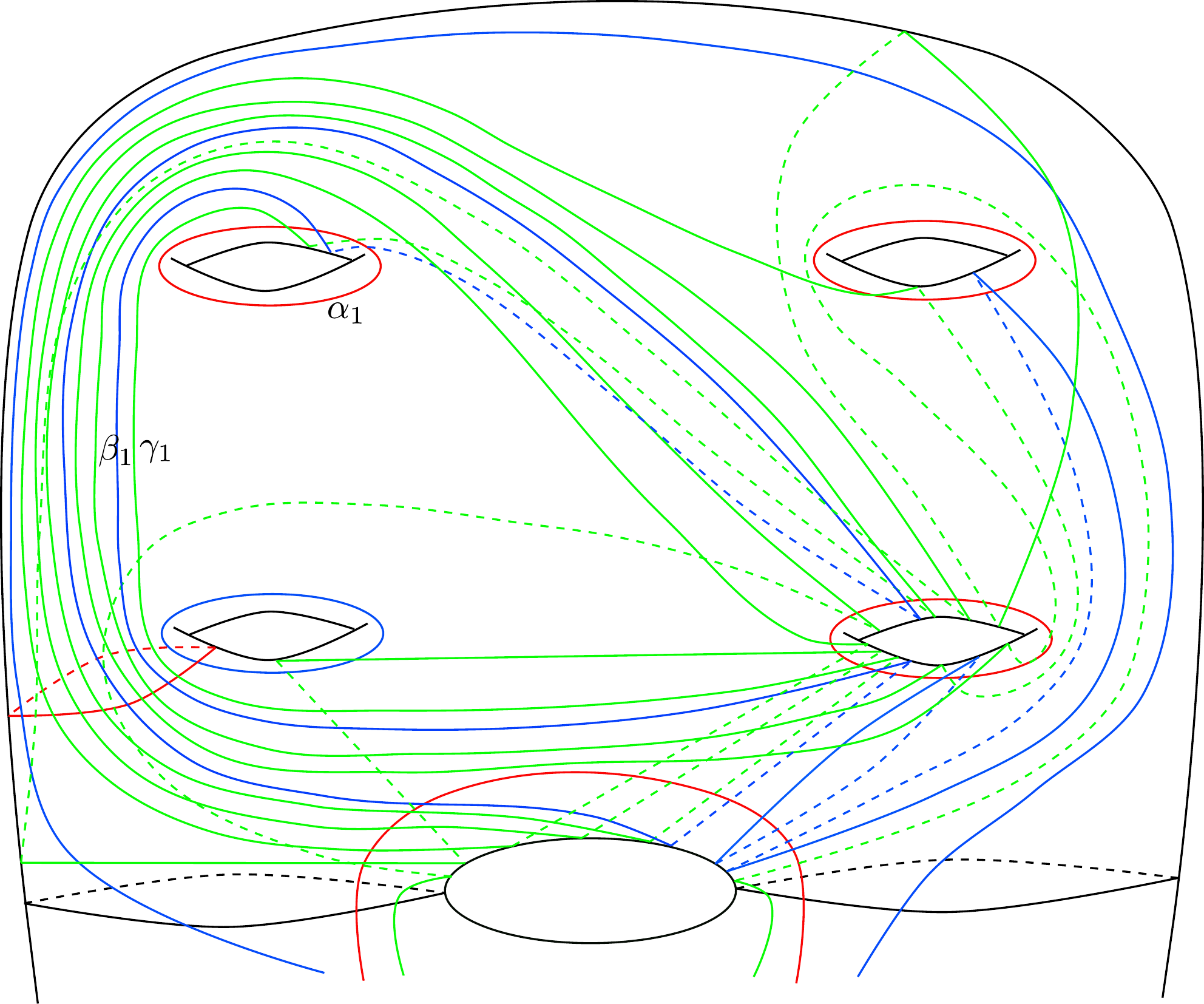}
\end{center}
\caption{Before the second destabilization.}
\label{fig:before second destabilization}
\end{figure}

\begin{figure}[h]
\begin{center}
\includegraphics[width=16cm, height=7cm, keepaspectratio, scale=1]{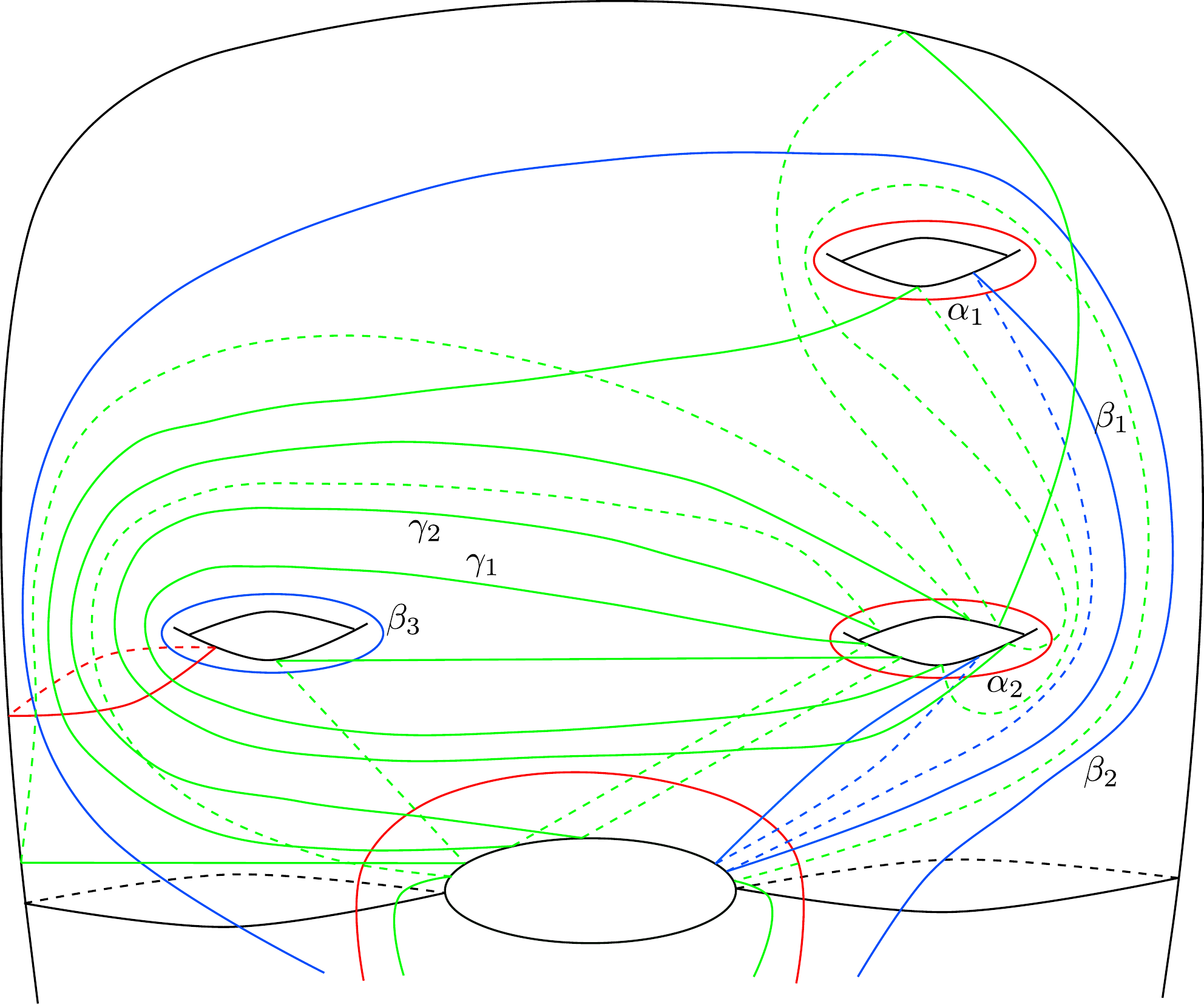}
\end{center}
\caption{After the second destabilization.}
\label{fig:after second destabilization}
\end{figure}

\subsection*{The third destabilization}
In Figure \ref{fig:after second destabilization} (or Figure \ref{fig:before third destabilization}), we will destabilize $\alpha_1$, $\beta_1$ and $\gamma_1$. To do this, we need to make $\beta_{1}$ parallel to $\gamma_1$. We slide $\gamma_1$ over $\gamma_2$ so that $\gamma_1$ does not intersect $\alpha_2$. Then, we slide $\beta_1$ over $\beta_2$ and $\beta_3$, so that $\beta_1$ is parallel to $\gamma_1$. As a result, we obtain Figure \ref{fig:before third destabilization}. In Figure \ref{fig:before third destabilization}, by destabilizing $\alpha_1$, $\beta_1$ and $\gamma_1$, we get Figure \ref{fig:after third destabilization}.

\begin{figure}[h]
\begin{center}
\includegraphics[width=16cm, height=7cm, keepaspectratio, scale=1]{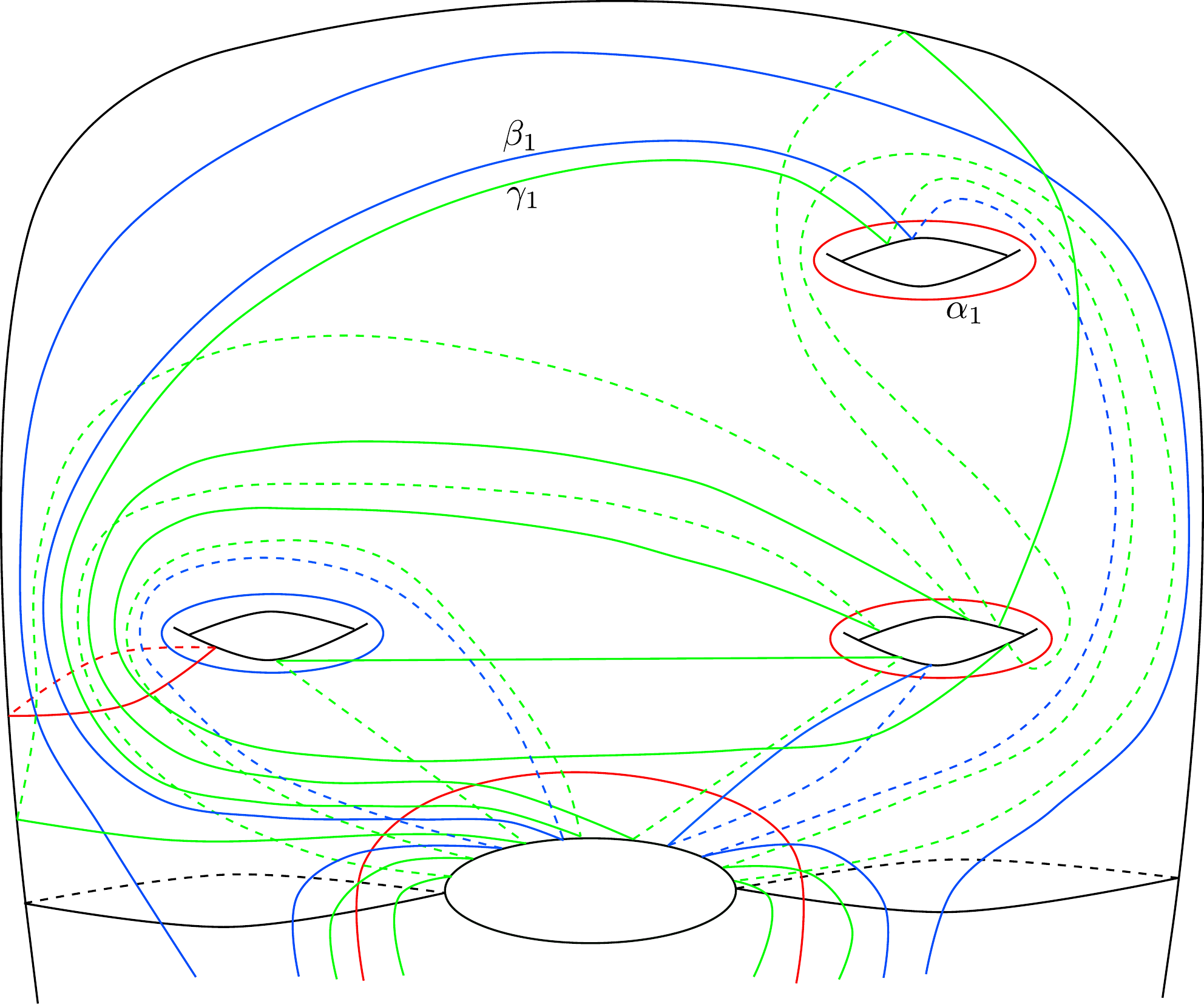}
\end{center}
\caption{Before the third destabilization.}
\label{fig:before third destabilization}
\end{figure}

\begin{figure}[h]
\begin{center}
\includegraphics[width=16cm, height=7cm, keepaspectratio, scale=1]{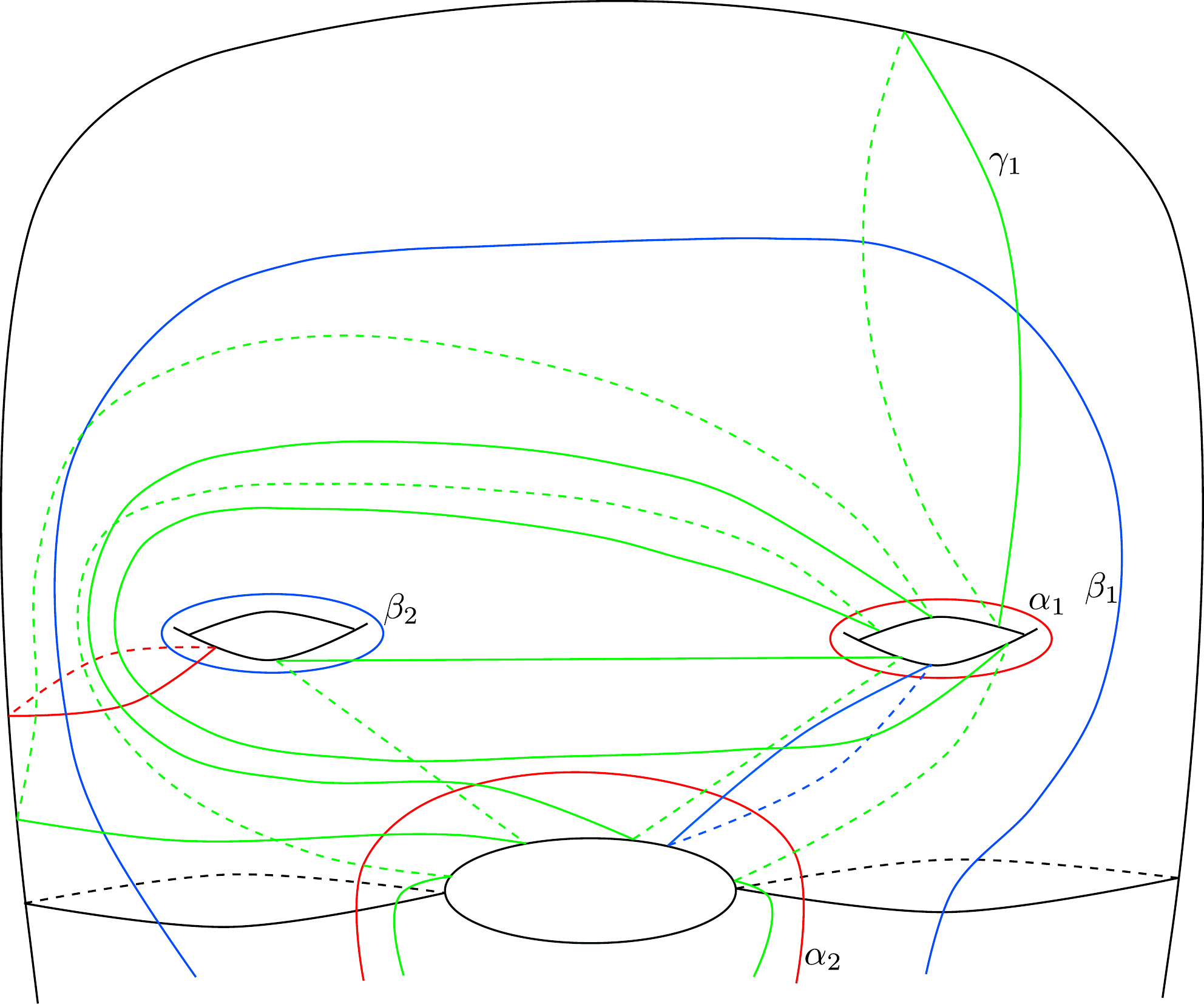}
\end{center}
\caption{After the third destabilization.}
\label{fig:after third destabilization}
\end{figure}

\subsection*{The fourth destabilization}
In Figure \ref{fig:after third destabilization} (or Figure \ref{fig:before fourth destabilization}), we will destabilize $\alpha_1$, $\beta_1$ and $\gamma_1$. To do this, we need to make $\beta_{1}$ parallel to $\alpha_1$. We slide $\beta_1$ over $\beta_2$ and $\alpha_1$ over $\alpha_2$, so that $\alpha_1$ is parallel to $\beta_1$. As a result, we obtain Figure \ref{fig:before fourth destabilization}. In Figure \ref{fig:before fourth destabilization}, by destabilizing $\alpha_1$, $\beta_1$ and $\gamma_1$, we get Figure \ref{fig:after fourth destabilization}.

\begin{figure}[h]
\begin{center}
\includegraphics[width=16cm, height=7cm, keepaspectratio, scale=1]{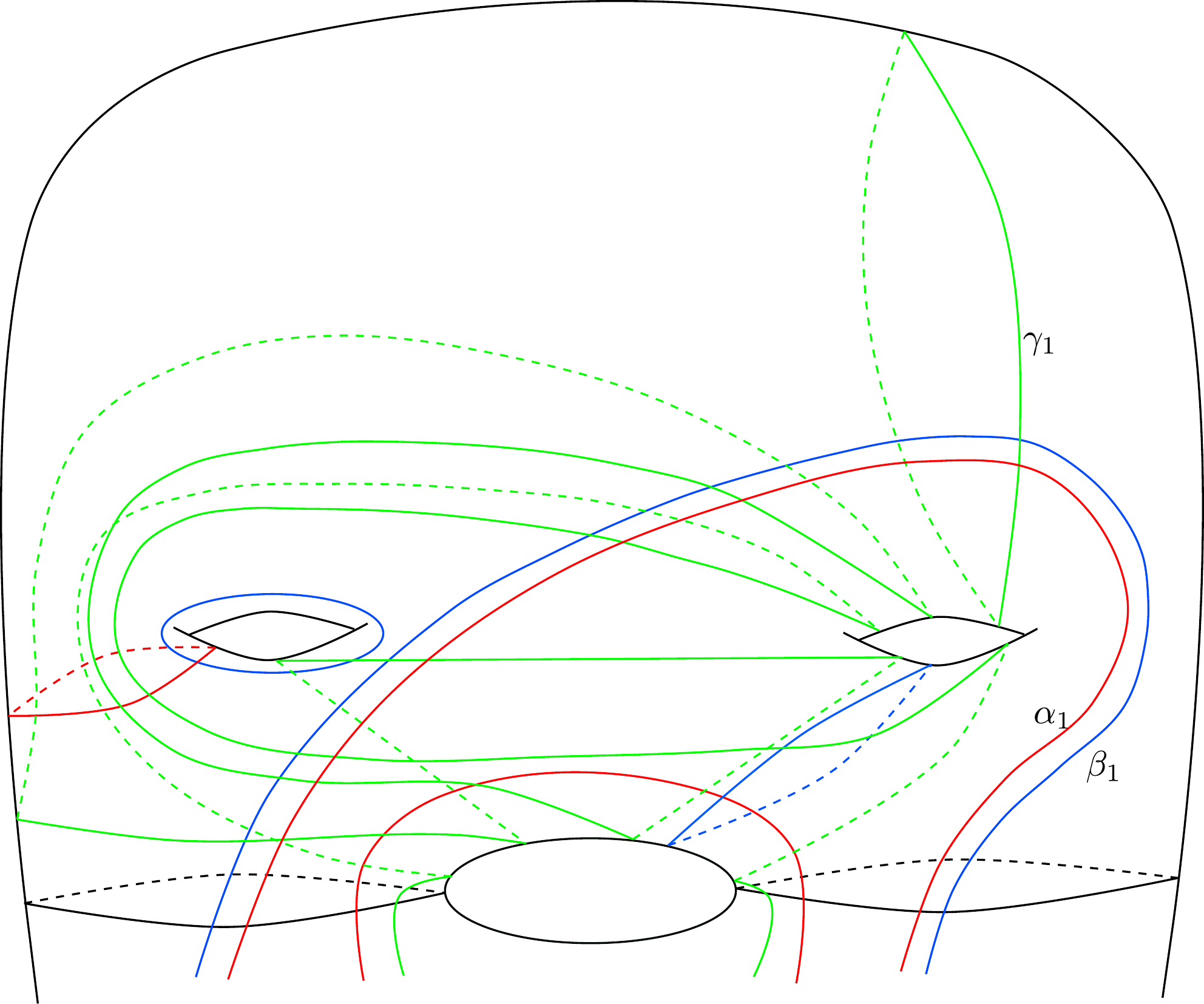}
\end{center}
\caption{Before the fourth destabilization.}
\label{fig:before fourth destabilization}
\end{figure}

\begin{figure}[h]
\begin{center}
\includegraphics[width=16cm, height=7cm, keepaspectratio, scale=1]{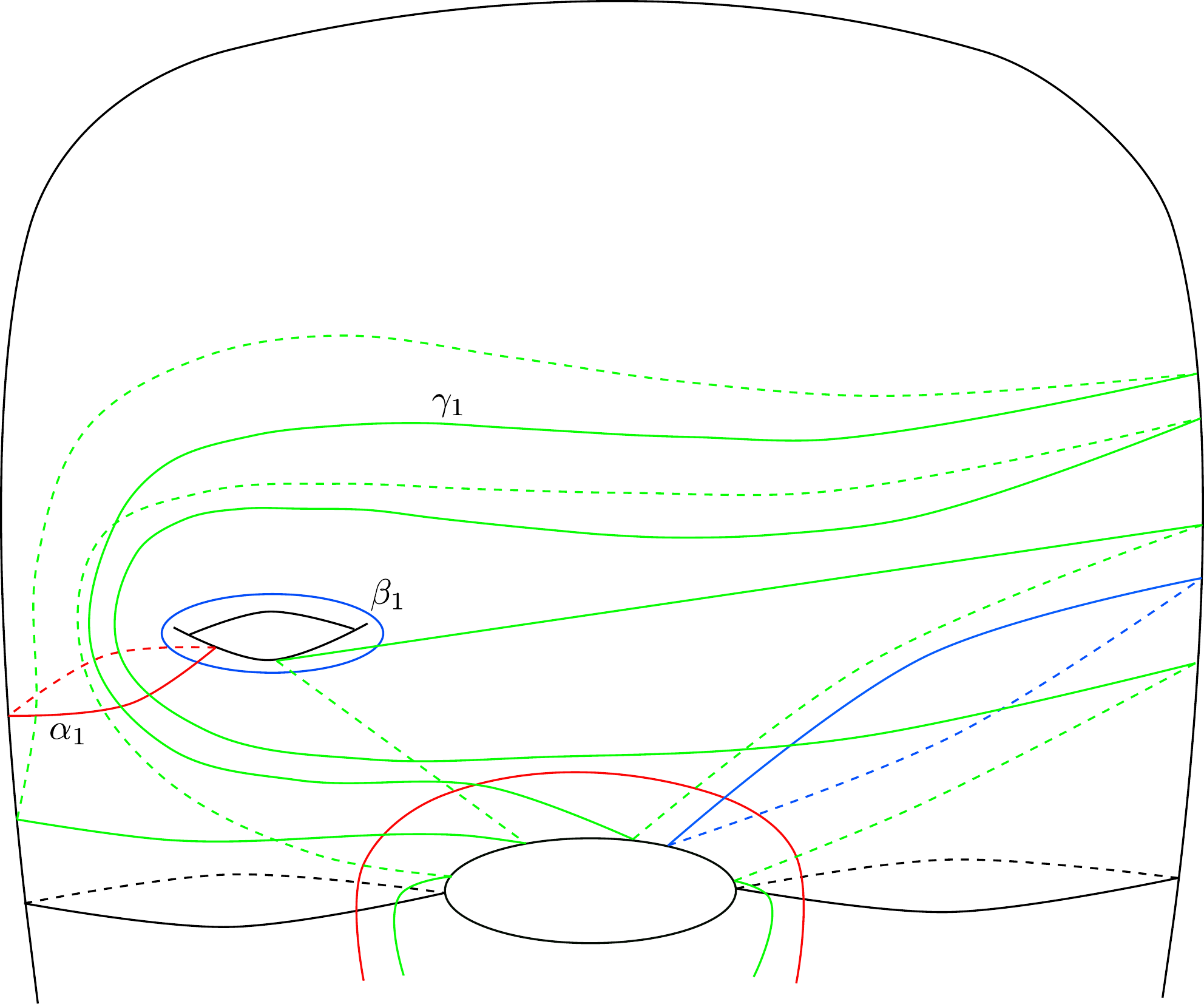}
\end{center}
\caption{After the fourth destabilization.}
\label{fig:after fourth destabilization}
\end{figure}

\subsection*{The fifth destabilization}
In Figure \ref{fig:after fourth destabilization}, we make $\gamma_1$ and $\alpha_1$ be parallel by isotopies. Then, we obtain Figure \ref{fig:before fifth destabilization}. In Figure \ref{fig:before fifth destabilization}, by destabilizing $\alpha_1$, $\beta_1$ and $\gamma_1$, we get Figure \ref{fig:after fifth destabilization}.

\begin{figure}[h]
\begin{center}
\includegraphics[width=16cm, height=7cm, keepaspectratio, scale=1]{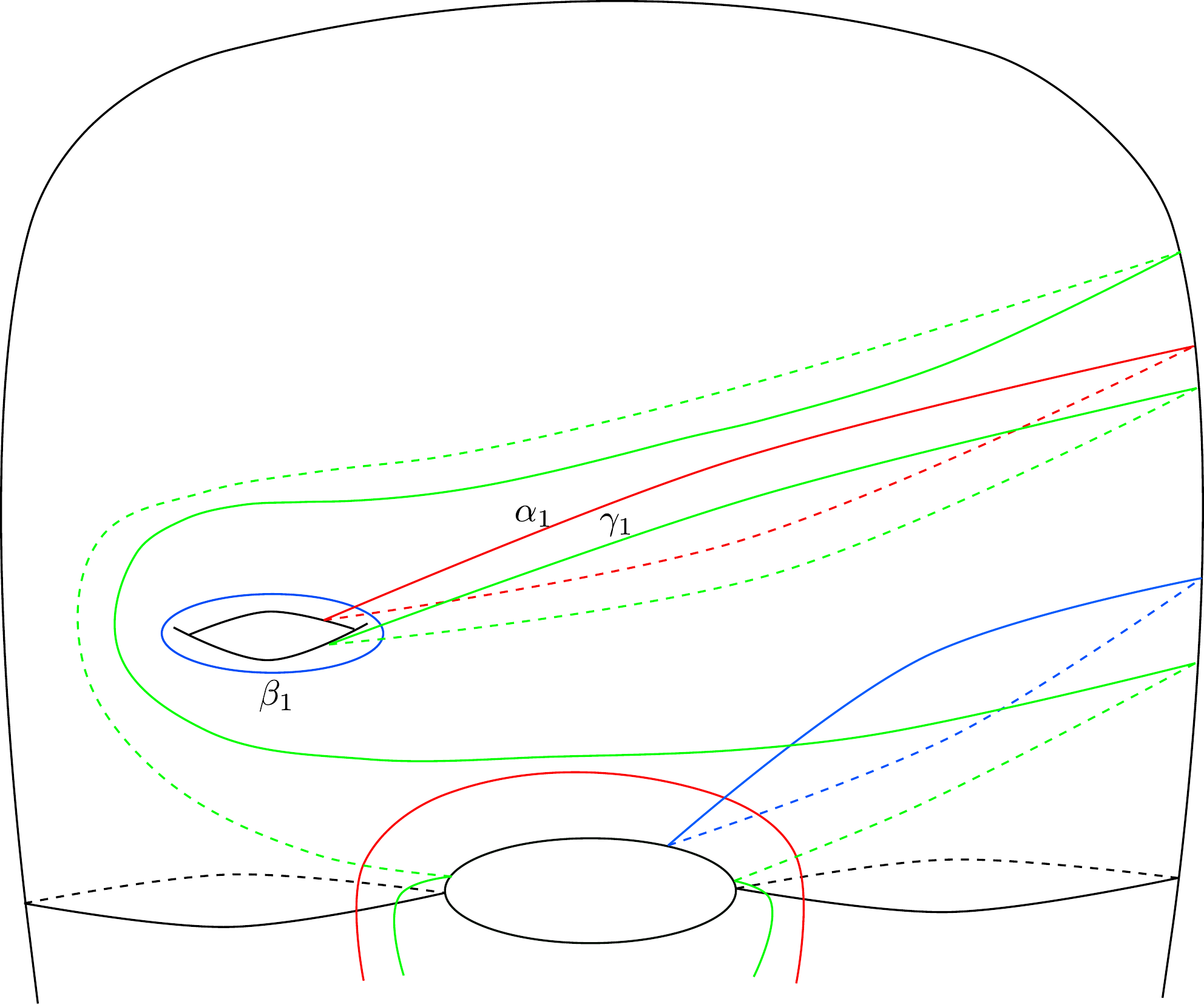}
\end{center}
\caption{Before the fifth destabilization.}
\label{fig:before fifth destabilization}
\end{figure}

\begin{figure}[h]
\begin{center}
\includegraphics[width=16cm, height=7cm, keepaspectratio, scale=1]{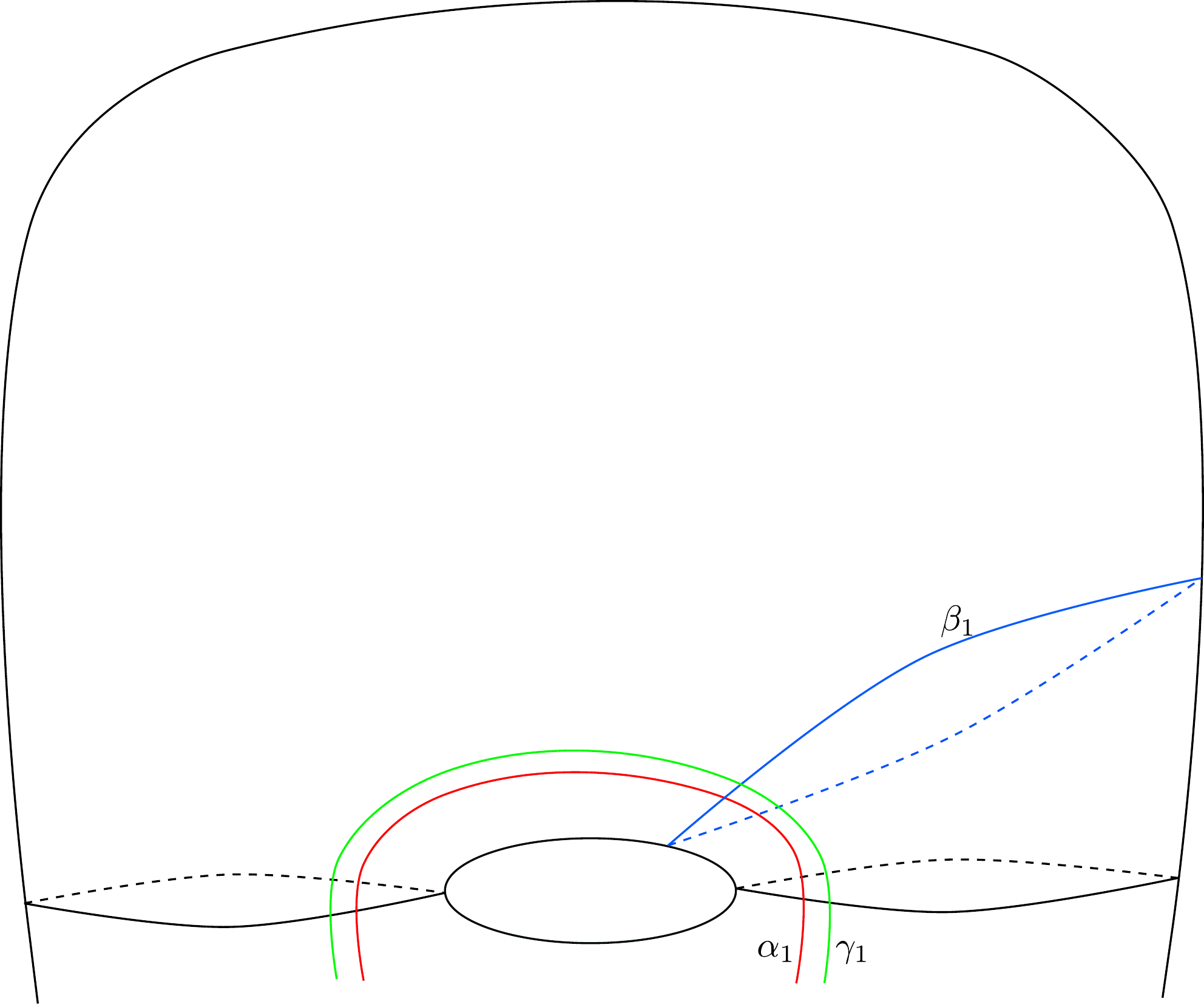}
\end{center}
\caption{After the fifth destabilization.}
\label{fig:after fifth destabilization}
\end{figure}

\subsection*{The sixth destabilization}
In Figure \ref{fig:after fifth destabilization}, the trisection of $X-\nu(K)$ is 0-annular since the normal Euler number of $K$ is 0. Thus, the monodromy of the open book decomposition is the identity, that is, $\alpha_1$ and $\gamma_1$ in Figure \ref{fig:after fifth destabilization} can be parallel. By destabilizing $\alpha_1$, $\beta_1$ and $\gamma_1$, we have a diagram $\mathcal{D}$.

The diagram $\mathcal{D}$ is obtained by attaching two disks to the two boundary components of the surface of $\mathcal{D}_{X-\nu(K)}$ since we surger along $\alpha_1$ when we destabilize $\alpha_1$, $\beta_1$ and $\gamma_1$ in Figure \ref{fig:after fifth destabilization}. In fact, $\mathcal{D}_{X-\nu(K)}$ is the diagram obtained by removing the open neighborhood of base points of the doubly pointed trisection diagram of $(X,K)$. Thus, $\mathcal{D}$ is the diagram obtained by simply deleting the base points. (Note that the surface erased the base points has no punctures.) In addition, the underlying trisection diagram of the doubly pointed trisection diagram of $(X,K)$ is the diagram of $T_{X}^{'}$. It can be seen from the way of boundary-stabilizations performed to construct a relative trisection diagram of $X-\nu(S)$ \cite{KM}. This means that $\mathcal{D}$ is just the diagram of $T_{X}^{'}$. Therefore, $T_S$ is diffeomorphic to a stabilization of $T_{X}^{'}$. Moreover, a meridionally stabilization of a bridge trisection corresponds to a stabilization for the underlying trisection. Thus, $T_{X}^{'}$ is a stabilization of $T_{X}$. This completes the proof of Theorem \ref{main theorem}.
\end{proof}


\begin{cor}\label{main corollary}
For each $P^2$-knot $S$ in $S^4$ that is of Kinoshita type, the trisection obtained by the trivial Price twist along $S$ is diffeomorphic to a stabilization of the genus 0 trisection of $S^4$.
\end{cor}

\begin{proof}
In Theorem \ref{main theorem}, if $X=S^4$, then $T_X$ is the genus 0 trisection of $S^4$ (see Remark \ref{rem:bridge trisection in $S^4$}).
\end{proof}

Lastly, as we have seen in Section \ref{sec:intro}, if any two diffeomorphic trisections of $S^4$ are isotopic, it follows from corollary \ref{main corollary} that the trisection obtained by the trivial Price twist along a $P^2$-knot which is of Kinoshita type is isotopic to a stabilization of the genus 0 trisection of $S^4$. Namely, Conjecture 3.11 in \cite{MSZ}, i.e. the conjecture that is a 4-dimansional analogue of Waldhausen's theorem on Heegaard splittings, is correct for this trisection.

\bibliographystyle{amsalpha}
\bibliography{math}

\end{document}